\documentclass[a4paper,11pt]{amsart}

\usepackage[utf8]{inputenc}		
\usepackage[T1]{fontenc}
\usepackage[english]{babel}

\usepackage{amsfonts}			
\usepackage{amsmath}
\usepackage{amssymb}
\usepackage{amsthm}

\usepackage{graphicx}			
\usepackage{tikz}
\usetikzlibrary{cd}

\usepackage{hyperref}			
\hypersetup{colorlinks=false, urlcolor=black, linkcolor=black}
\usepackage{cleveref}

\usepackage{enumitem}			

\usepackage{color}				
\usepackage{float}

\usepackage{thmtools}
\usepackage{thm-restate}

\usepackage{mathrsfs}

\newcommand{\Z}{\mathbb{Z}}						
\newcommand{\R}{\mathbb{R}}						
\renewcommand{\S}{\mathbb{S}}					
\newcommand{\B}{\mathbb{B}}

\newcommand{\eps}{\varepsilon}					
\renewcommand{\phi}{\varphi}					

\newcommand{\dd}								
{\mathop{}\!\mathrm{d}}						
\newcommand{\ddn}[1]							
{\mathop{}\!\mathrm{d^{#1}}}

\newcommand{\abs}[1]							
{\left| #1 \right|}
\newcommand{\smallabs}[1]						
{\lvert #1 \rvert}	
\newcommand{\norm}[1]							
{\left\lVert #1 \right\rVert}	
\newcommand{\smallnorm}[1]						
{\lVert #1 \rVert}						
\newcommand{\ip}[2]								
{\left< #1 , #2 \right>}

\DeclareMathOperator{\intr}{int}				
\DeclareMathOperator{\vol}{vol}					
\DeclareMathOperator{\spt}{spt}					
\DeclareMathOperator{\diam}{diam}				
\DeclareMathOperator*{\esssup}{ess\,sup}		

\DeclareMathOperator{\dist}{dist}

\newcommand{\iprodl}{\mathbin{\llcorner}}

\newcommand{\derham}{\mathrm{dR}}						
\newcommand{\hodge}{\mathtt{\star}\hspace{1pt}}

\newcommand{\loc}{\mathrm{loc}}

\newcommand{\mass}{\mathbf{M}}

\newcommand{\cX}{\mathcal{X}}

\newcommand{\ff}{\mathcal{F}_{K,D}(M,\omega)} 
\newcommand{\weakto}{\rightharpoonup}

\def\Xint#1{\mathchoice
	{\XXint\displaystyle\textstyle{#1}}%
	{\XXint\textstyle\scriptstyle{#1}}%
	{\XXint\scriptstyle\scriptscriptstyle{#1}}%
	{\XXint\scriptscriptstyle\scriptscriptstyle{#1}}%
	\!\int}
\def\XXint#1#2#3{{\setbox0=\hbox{$#1{#2#3}{\int}$}
		\vcenter{\hbox{$#2#3$}}\kern-.5\wd0}}

\def\dashint{\Xint-}

\newtheorem{thm}{Theorem}[section]{\bf}{\it}
\newtheorem{lemma}[thm]{Lemma}
\newtheorem{prop}[thm]{Proposition}
\newtheorem{cor}[thm]{Corollary}

{\bf}{\it}

{\bf}{\it}

{\bf}{\it}

{\bf}{\it}

\theoremstyle{definition}

\theoremstyle{remark}
\newtheorem{rem}[thm]{Remark}

\numberwithin{equation}{section}

\begin{document}
	
	\title{Quasiregular values and cohomology}
	
	\author[S. Heikkil\"a]{Susanna Heikkil\"a}
	\address{Department of Mathematics and Statistics, 
			P.O. Box 35 (MaD), FI-40014 University of Jyväskylä, Finland.
	}
	\email{susanna.a.heikkila@jyu.fi}
	
	\author[I. Kangasniemi]{Ilmari Kangasniemi}
	\address{Department of Mathematics and Statistics, 
			P.O. Box 35 (MaD), FI-40014 University of Jyväskylä, Finland.}
	\email{ilmari.k.kangasniemi@jyu.fi}

	\date{\today}
	\thanks{I. Kangasniemi was partially supported by the National Science Foundation grant DMS-2247469. S. Heikkilä was partially supported by the Research Council of Finland projects \#360505 and \#332671.}
	
	\begin{abstract}
        We prove that the recently shown cohomological obstruction for quasiregular ellipticity has a generalization in the theory of quasiregular values. More specifically, if $M$ is a closed, connected, and oriented Riemannian $n$-manifold, and there exists a map $f \in C(\mathbb{R}^n, M) \cap W^{1,n}_\loc(\mathbb{R}^n, M)$ satisfying $\lvert Df(x) \rvert^n \le K J_f(x) + \operatorname{dist}^n(f(x), f(x_0)) \Sigma(x)$ a.e.\ in $\mathbb{R}^n$ with $K \ge 1$, $x_0 \in \mathbb{R}^n$, and $\Sigma \in L^1(\mathbb{R}^n) \cap L^{1+\varepsilon}_{\text{loc}}(\mathbb{R}^n)$ for some $\varepsilon > 0$, then the real singular cohomology ring $H^*(M; \mathbb{R})$ of $M$ embeds into the exterior algebra $\wedge^* \mathbb{R}^n$ in a graded manner. We also show a partial version of our result for $M$ with dimension greater than $n$, by using a class of maps that combines properties of quasiregular values and quasiregular curves.
	\end{abstract}
	
	\maketitle 
	
	\section{Introduction}
	
	Let $N$ and $M$ be oriented Riemannian $n$-manifolds without boundary, where we assume all Riemannian manifolds in this paper to be $C^\infty$-smooth unless otherwise stated. A continuous locally Sobolev map $f \in C(N, M) \cap W^{1,n}_\loc(N, M)$ is called \emph{$K$-quasiregular} for a given value $K \ge 1$ if it satisfies the distortion inequality
	\begin{equation}\label{eq:QR_def}
		\abs{Df(x)}^n \le K J_f(x)
	\end{equation}
	at a.e.\ $x \in N$ with respect to the Riemannian volume measure. Here, $\abs{Df(x)}$ is the operator norm of the derivative map $Df(x) \colon T_x N \to T_{f(x)} M$ with respect to the Riemannian metrics of $N$ and $M$, and $J_f$ is the Jacobian determinant of $f$, defined a.e.\ on $N$ by $f^* \vol_M = J_f \vol_N$. A map $f$ is called \emph{quasiregular} if it is $K$-quasiregular for some $K \ge 1$.
	
	The theory of quasiregular mappings forms a higher dimensional counterpart to the geometric theory of holomorphic maps of a single complex variable. Indeed, quasiregular mappings satisfy counterparts to many of the results of classical complex analysis, such as the Liouville theorem \cite{Reshetnyak_Liouville}, the open mapping theorem \cite{Reshetnyak_Theorem2}, and the Picard theorem \cite{Rickman_Picard}. For standard reference texts on quasiregular maps, see e.g.\ \cite{Rickman_book, Iwaniec-Martin_book, Reshetnyak-book}.  
	
	A closed (i.e.\ compact without boundary), connected, oriented Riemannian $n$-manifold $M$ is called \emph{quasiregularly elliptic} if there exists a non-constant quasiregular map $f \colon \R^n \to M$. The study of quasiregularly elliptic manifolds can be traced back to a question of Gromov \cite[p. 200]{Gromov_Nilpotent} and Rickman \cite{Rickman_QREllipt_question} on whether every simply connected $M$ is quasiregularly elliptic. This question was resolved by Prywes \cite{Prywes_Annals}, who showed that if $M$ is quasiregularly elliptic, then its real singular cohomology spaces satisfy the dimension constraint
	\begin{equation}\label{eq:cohom_bound}
		\dim H^k(M; \R) \le \binom{n}{k}.
	\end{equation}
	This bound was previously conjectured by Bonk and Heinonen \cite{Bonk-Heinonen_Acta}, who had proven a weaker version with the upper bound dependent on the distortion $K$ of the map $f \colon \R^n \to M$.
	
	Afterwards, Prywes' result was refined by the first named author and Pankka \cite{Heikkila-Pankka_Elliptic}, who showed that if $M$ is quasiregularly elliptic, then there exists a graded embedding of algebras $\iota \colon H^*(M; \R) \to \wedge^* \R^n$ which maps the cup product of $H^*(M; \R)$ to the wedge product of the standard exterior algebra $\wedge^* \R^n$. We note that \eqref{eq:cohom_bound} is an immediate corollary of this embedding result, as $\dim \wedge^k \R^n = \binom{n}{k}$. We also note that embedding results of similar spirit were proven by the second named author \cite[Theorem 1.6]{Kangasniemi_Conf-formal} for uniformly quasiregularly elliptic manifolds, and by Berdnikov, Guth, and Manin \cite[Theorem 2.3]{Berdnikov-Guth-Manin} for manifolds admitting a Lipschitz mapping of positive asymptotic degree.
    
    Together with constructions by Piergallini and Zuddas \cite{Piergallini-Zuddas_Ellipticity}, the embedding result of \cite{Heikkila-Pankka_Elliptic} led into a full topological characterization of quasiregularly elliptic simply connected 4-manifolds. This topological characterization was later extended by Manin and Prywes \cite[Corollary 3.6]{Manin-Prywes} to also cover 4-manifolds with non-trivial fundamental groups.
    
	In this article, we show that the embedding theorem of \cite{Heikkila-Pankka_Elliptic} can be generalized using the recently formulated theory of quasiregular values. Here, if $N$ and $M$ are oriented Riemannian $n$-manifolds without boundary, $y_0 \in M$, $K \ge 0$, and $\Sigma \in L^1_\loc(N)$, we say that a map $f \in W^{1,n}_\loc(N, M)$ has a \emph{$(K, \Sigma)$-quasiregular value at $y_0$} if it satisfies the generalized distortion estimate
	\begin{equation}\label{eq:QRval_def}
		\abs{Df(x)}^n \le K J_f(x) + \dist^n(f(x), y_0) \Sigma(x) \quad \text{for a.e.\ }x \in N.
	\end{equation}
	With this definition, our main theorem is as follows.
	
	\begin{thm}\label{thm:QRvalue_embedding}
		Let $n \ge 2$, and let $M$ be a closed, connected, oriented Riemannian $n$-manifold. Suppose that there exists a non-constant map $f \in C(\R^n, M) \cap W^{1,n}_\loc(\R^n, M)$ that has a $(K, \Sigma)$-quasiregular value in its image with $K \ge 0$ and $\Sigma \in L^1(\R^n) \cap L^{1+\eps}_\loc(\R^n)$, where $\eps>0$. Then there exists a graded embedding of algebras
		\[
			\iota \colon H^*(M; \R) \to \wedge^* \R^n
		\]
		which maps the cup product of $H^*(M; \R)$ to the wedge product of $\wedge^* \R^n$. Consequently, for every $k \in \{0, \dots, n\}$, we have
		\[
			\dim H^k(M; \R) \le \binom{n}{k}.
		\]
	\end{thm}
	
	\subsection{Background on quasiregular values}
	
	The theory of quasiregular values was recently established by the second named author and Onninen in \cite{Kangasniemi-Onninen_1ptReshetnyak}. The origins of the theory can be traced back to a question of Astala, Iwaniec, and Martin \cite[Section 8.5]{Astala-Iwaniec-Martin_Book}, which, in current terminology, asked whether the constant function $f \equiv 0$ is the only map $f \in W^{1,n}_\loc(\R^n, \R^n)$ for which $\lim_{x \to \infty} f(x) = 0$ and $f$ has a $(K, \Sigma)$-quasiregular value at $0$ with $K \ge 1$ and $\Sigma \in L^{1+\eps}(\R^n) \cap L^{1-\eps}(\R^n)$, $\eps > 0$. A corresponding result in the planar case $n=2$ was used by Astala and P\"aiv\"arinta \cite[Prop.\ 3.3]{Astala-Paivarinta} in the solution of the planar Calder\'on problem. The higher-dimensional question was resolved in \cite{Kangasniemi-Onninen_Heterogeneous,Kangasniemi-Onninen_Heterogeneous-corrigendum}.
	
	The definition of quasiregular values can be compared and contrasted with the weaker condition
	\begin{equation}\label{eq:QR_with_Sigma}
		\abs{Df(x)}^n \le K J_f(x) + \Sigma(x) \qquad \text{for a.e.\ }x \in N,
	\end{equation}
	which can similarly be stated for mappings $f \in W^{1,n}_\loc(N, M)$ between oriented Riemannian $n$-manifolds. Condition \eqref{eq:QR_with_Sigma} is relatively far from the usual behavior of quasiregular maps, as it is satisfied by every Sobolev map $f \in W^{1,n}_\loc(N, M)$ with $K = 1$ and $\Sigma(x) = \abs{Df(x)}^n + \abs{J_f(x)}$. In contrast, the definition \eqref{eq:QRval_def} of a $(K, \Sigma)$-quasiregular value at $y_0$ allows for behavior similar to an arbitrary Sobolev map when $f(x)$ is far from $y_0$, but forces $f$ to behave in a more quasiregular manner as $f(x)$ approaches $y_0$.
	
	This behavior of mappings with $(K, \Sigma)$-quasiregular values leads to them satisfying single-value versions of the classical results of quasiregular maps at $y_0$, given sufficient integrability assumptions on $\Sigma$. These single-value results include a Liouville theorem \cite[Theorem 1.2]{Kangasniemi-Onninen_Heterogeneous}, a version of Reshetnyak's open mapping theorem \cite[Theorem 1.2]{Kangasniemi-Onninen_1ptReshetnyak}, and a version of Rickman's Picard theorem \cite[Theorem 1.2]{Kangasniemi-Onninen_Picard}. Moreover, a rescaling theorem for maps with quasiregular values was shown in \cite[Theorem 1.1]{Kangasniemi-Onninen_Rescaling}, which allows one to rescale such maps to $K$-quasiregular maps at a point $x_0 \in f^{-1}\{y_0\}$.
	
	We note that the statement of Theorem \ref{thm:QRvalue_embedding} is very close in spirit to the Liouville theorem \cite[Theorem 1.2]{Kangasniemi-Onninen_Heterogeneous} of quasiregular values. This theorem states that if $B \subset \R^n$ is bounded and $f \in C(\R^n, B) \cap W^{1,n}_\loc(\R^n, B)$ has a $(K, \Sigma)$-quasiregular value at $y_0 \in B$ with $\Sigma \in L^1(\R^n) \cap L^{1+\eps}_\loc(\R^n)$, $\eps > 0$, then either $y_0 \notin f(\R^n)$ or $f$ is identically $y_0$. Following some rewording, the statement of Theorem \ref{thm:QRvalue_embedding} is identical, but with the bounded set $B \subset \R^n$ replaced with a closed, connected, oriented Riemannian $n$-manifold $M$ that fails to have an embedding of algebras $H^*(M; \R) \hookrightarrow \wedge^* \R^n$.
	
	\subsection{Quasiregular curves}\enlargethispage{-2\baselineskip}
	
	Recently, the embedding theorem of \cite{Heikkila-Pankka_Elliptic} has also been generalized by the first named author in another direction, namely the theory of quasiregular curves; see \cite[Theorem 1.2]{Heikkila_Embedding-curves}. Quasiregular curves were introduced by Pankka \cite{Pankka_QRcurves} as a means to define quasiregular maps between spaces of different dimensions, generalizing e.g.\ the theories of holomorphic and pseudoholomorphic curves. For the definition of a quasiregular curve into an $m$-dimensional oriented Riemannian manifold $M$, one fixes a closed, non-vanishing differential $n$-form $\omega \in C^\infty(\wedge^n T^* M)$ on $M$ which essentially acts as a replacement for the volume form of an $n$-manifold. Then, given an oriented Riemannian $n$-manifold $N$, a map $F \in W^{1,n}_\loc(N, M)$ is called a \emph{$K$-quasiregular $\omega$-curve} if it satisfies
	\begin{equation}\label{eq:QRcurve_def}
		\abs{\omega_{F(x)}}_{\mass} \abs{DF(x)}^n \leq K \hodge (F^*\omega)_x \qquad \text{ for a.e.\ } x \in N.
	\end{equation}
	Here, $\hodge$ denotes the Hodge star operator, and the norm $\abs{\cdot}_{\mass}$ is the comass norm on $k$-covectors, as opposed to the usual Grassmannian norm $\abs{\cdot}$; see e.g.\ \cite[Section 1.8]{Federer_book}.
	
	In order to state the cohomological obstruction for quasiregular curves, we recall that for a smooth $m$-manifold $M$, a de Rham cohomology class $c \in H_\derham^n(M)$ belongs to the $n$:th layer $K^n(M)$ of the \emph{K\"unneth ideal} $K^*(M)$ of $M$ if $c = c_1 \wedge c_2 + \cdots + c_l \wedge c_{l+1}$ for some $c_i \in  H_\derham^{k_i}(M)$ and $c_{i+1} \in H_\derham^{n-k_i}(M)$ with $1 \le k_i \le n-1$ for $i=1,\ldots,l$. The main result of \cite{Heikkila_Embedding-curves} then states that if $M$ is a closed, connected, oriented Riemannian $m$-manifold, $\omega \in C^\infty(\wedge^n T^* M)$ is a closed non-vanishing $n$-form on $M$ for which the de Rham class $[\omega]$ is in $K^n(M)$, and $F \in W^{1,n}_\loc(\R^n, M)$ is a $K$-quasiregular $\omega$-curve with $\norm{DF}_{L^n(\R^n)} = \infty$, then there exists a graded homomorphism of algebras $\iota \colon H_\derham^*(M) \to \wedge^* \R^n$ for which $\iota([\omega]) \ne 0$. By a recent result of Ikonen \cite[Corollary 1.15]{Ikonen_QRCurves-removability}, the assumption $\norm{DF}_{L^n(\R^n)} = \infty$ can be eliminated. Note that in the case $m = n$, Poincar\'e duality automatically shows that a map $\iota$ as above is injective.
	
	In our case, the methods we employ in fact yield a combination of Theorem \ref{thm:QRvalue_embedding} and the main result of \cite{Heikkila_Embedding-curves}. For this, given two oriented Riemannian manifolds $N$, $M$ of dimensions $n$ and $m$, respectively, and a closed non-vanishing $n$-form $\omega \in C^\infty(\wedge^n T^* M)$ on $M$, we consider a combination of \eqref{eq:QR_with_Sigma} and \eqref{eq:QRcurve_def} given by
	\begin{equation}\label{eq:QR_curve_with_Sigma}
		\abs{\omega_{F(x)}}_{\mass} \abs{DF(x)}^n \leq K \hodge (F^*\omega)_x + \Sigma(x) \qquad \text{ for a.e.\ } x \in N,
	\end{equation}
	where $F \in W^{1,n}_\loc(N, M)$, $K \ge 0$, and $\Sigma \in L^1_\loc(N)$. The result we obtain for such maps is as follows.
	
	\begin{thm}\label{thm:QRcurve_with_sigma_embedding}
		Let $m \ge n \ge 2$, let $M$ be a closed, connected, oriented Riemannian $m$-manifold, let $\omega \in C^\infty(\wedge^n T^* M)$ be a closed non-vanishing $n$-form on $M$ for which the de Rham class $[\omega]$ is in the K\"unneth ideal $K^*(M)$, let $K \ge 0$, and let $\Sigma \in L^1(\R^n)$. If there exists a map $F \in C(\R^n, M) \cap W^{1,n}_\loc(\R^n, M)$ for which $\norm{DF}_{L^n(\R^n)} = \infty$ and $F$ satisfies \eqref{eq:QR_curve_with_Sigma} with $K$, $\Sigma$, and $\omega$, then there exists a graded homomorphism of algebras 
		\[
			\iota \colon H_\derham^*(M) \to \wedge^* \R^n
		\]
		for which $\iota([\omega]) \ne 0$.
	\end{thm}
	
	One can also consider an analogous combination of \eqref{eq:QRval_def} and \eqref{eq:QRcurve_def} given by
	\begin{equation}\label{eq:QRcurveval_def}
		\abs{\omega_{F(x)}}_{\mass} \abs{DF(x)}^n \leq K \hodge (F^*\omega)_x + \dist^n(F(x), y_0)\Sigma(x) \text{ for a.e.\ } x \in N
	\end{equation}
	with $y_0 \in M$. As the manifold $M$ in Theorem \ref{thm:QRcurve_with_sigma_embedding} is bounded, the map $\dist(\cdot, y_0)$ is bounded for every $y_0 \in M$. Thus, a version of Theorem \ref{thm:QRcurve_with_sigma_embedding} for \eqref{eq:QRcurveval_def} follows immediately.
	
	\begin{cor}\label{cor:QRcurveval_embedding}
		Let $m \ge n \ge 2$, let $M$ be a closed, connected, oriented Riemannian $m$-manifold, let $\omega \in C^\infty(\wedge^n T^* M)$ be a closed non-vanishing $n$-form on $M$ for which the de Rham class $[\omega]$ is in the K\"unneth ideal $K^*(M)$, let $K \ge 0$, let $y_0 \in M$, and let $\Sigma \in L^1(\R^n)$. If there exists a map $F \in C(\R^n, M) \cap W^{1,n}_\loc(\R^n, M)$ for which $\norm{DF}_{L^n(\R^n)} = \infty$ and $F$ satisfies \eqref{eq:QRcurveval_def} with $K$, $\Sigma$, and $\omega$, then there exists a graded homomorphism of algebras 
		\[
		\iota \colon H_\derham^*(M) \to \wedge^* \R^n
		\]
		for which $\iota([\omega]) \ne 0$.
	\end{cor}
	
	\subsection{Methods used in the proofs}
	
	The proof of Theorem \ref{thm:QRcurve_with_sigma_embedding} and Corollary \ref{cor:QRcurveval_embedding} follows the same strategy that has been developed in \cite{Heikkila-Pankka_Elliptic} and \cite{Heikkila_Embedding-curves}, with some adjustments required by the introduction of the $\Sigma$-term. In fact, a large portion of this proof only uses properties of general $L^p$-bounded pull-back operators on differential forms, and hence does not directly invoke \eqref{eq:QR_curve_with_Sigma} or quasiregularity in general. We have abstracted this part of the proof into a separate technical result, Proposition \ref{prop:technical_abstraction}. Our statement of this technical abstraction is sufficiently general that it also yields the corresponding embedding theorem for Lipschitz maps of positive asymptotic degree as a short corollary, as we outline in Section \ref{subsect:abstraction}; see also the comparison of these two embedding results by Manin and Prywes in \cite[Section 2]{Manin-Prywes}.
	
	Besides Corollary \ref{cor:QRcurveval_embedding}, the second component of the proof of Theorem \ref{thm:QRvalue_embedding} is the following result.
	
	\begin{prop}\label{prop:QRval_infinite_energy}
		Let $M$ be a closed, connected, oriented Riemannian $n$-manifold, and let $f \in C(\R^n, M) \cap W^{1,n}_\loc(\R^n, M)$. Suppose that $M$ is not a rational cohomology sphere, and that $f$ has a $(K, \Sigma)$-quasiregular value at $y_0 \in M$, where $K \ge 1$ and $\Sigma \in L^1(\R^n) \cap L^{1+\eps}_\loc(\R^n)$ with $\eps >0$. If $\norm{Df}_{L^n(\R^n)} < \infty$, then $y_0 \notin f(\R^n)$.
	\end{prop}
	
	Proposition \ref{prop:QRval_infinite_energy} ensures that in all non-trivial cases of Theorem \ref{thm:QRvalue_embedding}, the mapping $f$ has $\norm{Df}_{L^n(\R^n)} = \infty$, and Corollary \ref{cor:QRcurveval_embedding} is hence applicable. For quasiregular maps and quasiregular curves, this step is achieved with \cite[Theorem 1.11]{Bonk-Heinonen_Acta} and \cite[Corollary 1.15]{Ikonen_QRCurves-removability}, respectively. However, our proof of Proposition \ref{prop:QRval_infinite_energy} differs significantly from the proofs of these prior results; see e.g.\ Remark \ref{rem:polynomial_growth} for some of the complications involved in attempting to apply prior methods. The proof we use relies heavily on the fact that the theory of quasiregular values has a version of Reshetynak's theorem. Due to this reliance on a Reshetnyak-type theorem, our approach appears limited to the case where $M$ is $n$-dimensional; see the discussion in \cite[Remark 1.11]{Pankka_QRcurves}.

    \subsection{Assumptions of continuity in Theorems \ref{thm:QRvalue_embedding} and \ref{thm:QRcurve_with_sigma_embedding}}

    We recall that in many cases, every map $f \in W^{1,n}_\loc(N, M)$ satisfying \eqref{eq:QR_def} has a continuous Sobolev representative, and thus the assumption of continuity in the definition of quasiregular maps only amounts to selecting the correct representative. For a version of this result with closed manifold targets, see the work of Goldstein, Haj{\l}asz, and Pakzad \cite{Goldstein-Hajlasz-Pakzad}.
    
    Similarly, if $\Sigma$ is locally $L^{1+\eps}$-integrable, it is often guaranteed that a mapping satisfying \eqref{eq:QRval_def} or \eqref{eq:QR_with_Sigma} has a continuous representative, due to a higher integrability argument for $\abs{Df}$ based on Gehring's lemma; see e.g.\ \cite[Section 2.1]{Dolezalova-Kangasniemi-Onninen_MGFD-cont} for the argument when the target is $\R^n$. In this article, we also show a version of this higher integrability result when the target of $f$ is a closed manifold $M$ that is not a rational cohomology sphere; thus, in all non-trivial cases of our main results, a continuous representative exists automatically. The version of this result for solutions of \eqref{eq:QR_curve_with_Sigma} or \eqref{eq:QRcurveval_def} is as follows.
	
	\begin{prop}\label{prop:higher_int_general}
		Let $m \ge n \ge 2$, let $M$ be a closed, connected, oriented Riemannian $m$-manifold, let $\Omega \subset \R^n$ be open, and let $F \in W^{1,n}(\Omega, M)$ satisfy \eqref{eq:QR_curve_with_Sigma} or \eqref{eq:QRcurveval_def}, where $K \ge 0$, $\omega \in C^\infty(\wedge^n T^* M)$ is closed and non-vanishing with $[\omega] \in K^*(M)$, and $\Sigma \in L^{1+\eps}_\loc(\Omega)$ for some $\eps > 0$. Then $F \in W^{1, n+\eps'}_\loc(\Omega, M)$ for some $\eps' > 0$, and consequently, $F$ has a continuous Sobolev representative. 
	\end{prop}
	
	The simpler version of the statement for $m = n$ is as follows.
	
	\begin{cor}\label{cor:higher_int_same_dimension}
		Let $n \ge 2$, let $M$ be a closed, connected, oriented Riemannian $n$-manifold, let $\Omega \subset \R^n$ be open, and let $f \in W^{1,n}(\Omega, M)$ satisfy \eqref{eq:QR_with_Sigma} or \eqref{eq:QRval_def}, where $K \ge 0$ and $\Sigma \in L^{1+\eps}_\loc(\Omega)$ for some $\eps > 0$. If $M$ is not a rational cohomology sphere, then $f \in W^{1, n+\eps'}_\loc(\Omega, M)$ for some $\eps' > 0$, and consequently, $f$ has a continuous Sobolev representative. 
	\end{cor}

    \subsection{The organization of this paper}
    In Section \ref{sect:prelims}, we recall the necessary preliminaries for the proofs of our main results. Some simple generalizations of known results are also shown here, such as a Riemannian manifold version of the Reshetnyak's theorem for quasiregular values. In Section \ref{sect:proof_of_cohom_embedding}, we provide the parts of the proof of Theorem \ref{thm:QRcurve_with_sigma_embedding} which are specific to \eqref{eq:QR_curve_with_Sigma}. In Section \ref{sect:proof_of_infinite_energy}, we prove Proposition \ref{prop:QRval_infinite_energy}. In Section \ref{sect:proof_of_continuity}, we prove Proposition \ref{prop:higher_int_general}. Finally, in Section \ref{sect:abstraction_proof}, we prove Proposition \ref{prop:technical_abstraction}, the technical result which abstracts the more general parts of the argument in \cite{Heikkila-Pankka_Elliptic} and \cite{Heikkila_Embedding-curves}.
	
	\subsection*{Acknowledgments}
	
	The authors thank Pekka Pankka for numerous discussions on the topics of the article. I.K.\ in particular thanks him for the invitation to visit University of Helsinki during Summer 2024, as this visit was highly valuable for the completion of this research project. The authors also thank Kai Rajala for several useful remarks on the presentation of this article, as well as Toni Ikonen for a helpful observation at the right time.
	
	\section{Preliminaries} \label{sect:prelims}
	
	We first specify some terminology. We say that a manifold $M$ is \emph{open} if it contains no boundary points. Thus, under our terminology, for instance $\R^n$, $\S^n$, and the open Euclidean unit ball $\B^n$ are open manifolds, but the closed Euclidean unit ball $\overline{\B^n}$ is not. We say that a manifold is \emph{closed} if it is compact and contains no boundary points. 
    
    If $p \in [1, \infty]$, we use $p^* = p/(p-1)$ to denote the H\"older conjugate of $p$. Moreover, if $r > 0$, we use $\B^n_r$ to the open Euclidean ball of radius $r$ centered at the origin.
	
	\subsection{Sobolev maps and Sobolev differential forms on manifolds}

	Let let $N$ be an oriented, Riemannian $n$-manifold, and let $M$ be an oriented, Riemannian $m$-manifold. Recall that for $k \in \{0, \dots, m \}$, a \emph{differential $k$-form} on $M$ is a section $\omega \colon M \to \wedge^k T^* M$ of the $k$:th exterior power of the cotangent bundle of $M$. We denote the value of a $k$-form $\omega$ at $x \in M$ by $\omega_x \in \wedge^k T^*_x M$; in particular, $\omega_x$ is an alternating multilinear map $(T_x M)^k \to \R$. 
	
	We use $C^\infty(\wedge^k T^* M)$ to denote the space of smooth differential $k$-forms on $M$, with $C^\infty_0(\wedge^k T^* M)$ used for smooth $k$-forms with compact support. We denote the exterior derivative by $d \colon C^\infty(\wedge^k T^* M) \to C^\infty(\wedge^{k+1} T^* M)$, and the wedge product of differential forms by $\wedge$. We also denote the pull-back of a differential $k$-form $\omega$ on $M$ by a smooth map $f \in C^\infty(N, M)$ by $f^*\omega$; i.e., $f^* \omega$ is a $k$-form on $N$ with 
	\begin{equation}\label{eq:k-form_pullback}
		f^*\omega_x(v_1, \dots, v_k) = \omega_{f(x)}(Df(x)v_1, \dots, Df(x) v_k)
	\end{equation}
	for all $x \in N$ and $v_i \in T_x N$.
	
	Note that due to being oriented and Riemannian, the manifolds $N$ and $M$ have canonical volume forms $\vol_N \in C^\infty(\wedge^n T^* N)$ and $\vol_M \in C^\infty(\wedge^m T^*M)$, which induce a Lebesgue measure on each manifold. We say that a differential $k$-form $\omega$ on $M$ is \emph{(Lebesgue) measurable} if for every diffeomorphic chart $\phi \colon U \to M$, $U \subset \R^m$, we have that $\phi^*\omega \colon U \to \wedge^k (\R^m)^*$ has (Lebesgue) measurable coefficient functions. For every $x \in M$, the Riemannian metric on $M$ induces a Grassmann inner product $\ip{\cdot}{\cdot}$ on every space $\wedge^k T^*_x M$, and a corresponding norm $\abs{\cdot}$ on $\wedge^k T^*_x M$. We obtain that for all measurable $k$-forms $\omega, \omega' \colon M \to \wedge^k T^*M$, the functions $\ip{\omega}{\omega'} \colon M \to \R$ and $\abs{\omega} \colon M \to [0, \infty)$ are measurable.
	
	We then let $p \in [1, \infty]$, and define the space $L^p(\wedge^k T^* M)$ of \emph{$L^p$-integrable differential $k$-forms} as the space of all measurable $k$-forms $\omega$ on $M$ with $\abs{\omega} \in L^p(M)$. This is a complete normed space when equipped with the norm $\norm{\omega}_{L^p(M)} := \norm{\abs{\omega}}_{L^p(M)}$, after applying the usual equivalence relation, where $\omega = \omega'$ if $\omega_x = \omega'_x$ outside a null-set of $x \in M$. We also define $L^p_\loc(\wedge^k T^* M)$ as the space of all measurable $k$-forms $\omega$ on $M$ with $\abs{\omega} \in L^p_\loc(M)$.
	
	If $\omega \in L^1_\loc(\wedge^k T^* M)$, then we say that a $(k+1)$-form $d\omega \in L^1_\loc(\wedge^{k+1} T^* M)$ is a \emph{weak exterior derivative} of $\omega$ if
	\[
		\int_M d\omega \wedge \eta = (-1)^{k+1} \int_M \omega \wedge d\eta
	\]
	for all test forms $\eta \in C^{\infty}_0(\wedge^{n-k-1} T^* M)$. We note that weak exterior derivatives are unique up to a null-set. We use $W^{d,p}(\wedge^k T^* M)$ to denote the space of all $\omega \in L^p(\wedge^k T^* M)$ with a weak exterior derivative $d\omega \in L^p(\wedge^{k+1} T^* M)$, and $W^{d,p}_0(\wedge^k T^* M)$ to denote the space of compactly supported elements of $W^{d,p}(\wedge^k T^* M)$. We also use $W^{d,p}_\loc(\wedge^k T^* M)$ to denote the local counterpart to this space: i.e., $\omega \in W^{d,p}_\loc(\wedge^k T^* M)$ if $\omega \in L^p_\loc(\wedge^k T^* M)$ and $\omega$ has a weak exterior derivative $d\omega \in L^p_\loc(\wedge^{k+1} T^* M)$.
	
	If $l \in \Z_{> 0}$, we define that a measurable map $f \colon N \to \R^l$ is in the local Sobolev space $W^{1,p}_\loc(N, \R^l)$ if its every coordinate function $f_i$ is in $W^{d,p}_\loc(\wedge^0 T^* N)$. We then note that there exists a Nash embedding of $M$ into a Euclidean space, i.e.\ an embedding $\iota \in C^\infty(M, \R^l)$ for some $l > 0$ such that $\iota$ is isometric in the sense of preserving the Riemannian metric. We define that a map $f \colon N \to M$ is in the space $W^{1,p}_\loc(N, M)$ if $\iota \circ f \in W^{1,p}_\loc(N, \R^l)$. The resulting space is independent of the choice of $\iota$. Moreover, a mapping $f \in W^{1,p}_\loc(N, \R^l)$ possesses a measurable a.e.\ unique weak derivative map $Df \colon TN \to TM$ with its point-wise operator norm function satisfying $\abs{Df} \in L^p_\loc(N)$. For details on manifold-valued Sobolev functions, we refer readers to e.g.\ \cite{Hajlasz-Iwaniec-Maly-Onninen} or the notes \cite{Kangasniemi_QRnotes}.
	
	Let then $\omega$ be a $k$-form on $M$, and $f \in W^{1,p}_\loc(N,M)$. With the measurable weak derivative $Df \colon TN \to TM$, we may define the pull-back $f^*\omega$ as in \eqref{eq:k-form_pullback}. We note that $f^* \omega$ satisfies a point-wise norm estimate, of the form
	\begin{equation}\label{eq:pull-back_pointw_est}
		\smallabs{(f^* \omega)_x} \le \smallabs{\omega_{f(x)}} \abs{Df(x)}^k
	\end{equation}
	for a.e.\ $x \in N$. Based on this estimate, one can show that if $\omega \in C^\infty_0(\wedge^k T^* M)$ and $f \in W^{1,p}_\loc(N, M)$, then $f^* \omega \in L^{p/k}_\loc(\wedge^k T^* N)$ and $f^*\omega$ has a weak exterior derivative $df^* \omega = f^* d\omega \in L^{p/(k+1)}_\loc(\wedge^k T^* N)$. A detailed exposition of this standard result can be found in the notes \cite[Proposition 9.14]{Kangasniemi_QRnotes}.
	
	\subsection{The degree and the local index}\label{sect:degree_index}
	
	Let $N$, $M$ be oriented, open Riemannian $n$-manifolds, and let $f \in C(N, M) \cap W^{1,n}_\loc(N, M)$. For every open $U \subset N$, for which $\overline{U}$ is compact, and every $y \in M \setminus f (\partial U)$, we define the \emph{degree} $\deg(f, y, U)$ of $f$ at $y$ with respect to $U$ as follows: we let $V$ be the connected component of $y$ in $M \setminus f (\partial U)$, and define
	\begin{multline*}
		\deg(f, y, U) = \int_U f^* \omega,\\
		\text{where }\omega \in C^\infty_0(\wedge^n T^* M), \text{ } \spt \omega \subset V, \text{ and } \int_V \omega = 1.
	\end{multline*}
	Note that here, $f^* \omega\vert_U \in L^1(\wedge^n T^*U)$ since $\omega \in C^\infty_0(\wedge^n T^* M)$ and $f\vert_U \in W^{1,n}(U, M)$, and thus the above integral is well-defined and finite. 
	
	For this definition to be sensible, we also need to point out that $\deg(f, y, U)$ is independent of the choice of $\omega$ above. Indeed, $V$ is a connected, oriented, open $n$-manifold, so by Poincar\'e duality and the de Rham theorem, its $n$:th compactly supported de Rham cohomology space satisfies $H^n_{\derham, 0}(V) \cong \R$. As integration over $V$ defines a surjective linear map $H^n_{\derham, 0}(V) \to \R$, it follows that if $\omega, \omega' \in C^\infty_0(\wedge^n T^* V)$ both have unit integral, then $\omega - \omega' = d\eta$ for some $\eta \in C^\infty_0(\wedge^{n-1} T^* M)$ with $\spt \eta \subset V$. Now $f^* \eta\vert_U \in W^{d,1}_\loc(\wedge^{n-1} T^* U)$, and since $\overline{U}$ is compact and $\partial V \subset f (\partial U)$, we also have that $\spt(f^* \eta\vert_U)$ is compact; that is, $f^* \eta\vert_U \in W^{d,1}_0(\wedge^{n-1} T^*U)$. It remains to approximate $f^* \eta\vert_U$ in the $W^{d,1}$-norm with smooth forms $\alpha_j \in C^\infty_0(\wedge^{n-1} T^* U)$ using diffeomorphic charts and convolutions, and to then use Stokes' theorem on $d\alpha_j$ to conclude that the integral of $(f^*\omega - f^*\omega')\vert_U = (f^* d\eta)\vert_U$ vanishes. Thus, the degree is indeed well-defined.
	
	We recall that the degree satisfies the following excision property.
	
	\begin{lemma}\label{lem:degree_excision}
		Let $N$, $M$ be oriented, open Riemannian $n$-manifolds, let $f \in C(N, M) \cap W^{1,n}_\loc(N, M)$, and let $U, V \subset N$ be open sets such that $\overline{V} \subset U$ and $\overline{U}$ is compact. Then for all $y \in M \setminus f(\overline{U} \setminus V)$, we have $\deg(f, y, U) = \deg(f, y, V)$.
	\end{lemma}
	\begin{proof}
		We select an $\omega \in C^\infty_0(\wedge^n T^* M)$ with unit integral that is supported in the connected component of $y$ in $M \setminus f(\overline{U} \setminus V)$. Then $\omega$ is a valid choice for the $n$-form in both of the definitions of $\deg(f, y, U)$ and $\deg(f, y, V)$. Since $f^* \omega$ vanishes on $U \setminus V$, it follows that the integrals of $f^* \omega$ over $U$ and $V$ agree, and thus $\deg(f, y, U) = \deg(f, y, V)$.
	\end{proof}
	
	Suppose then that $x \in N$ is an isolated point of $f^{-1}\{f(x)\}$. It follows from Lemma \ref{lem:degree_excision} that for every open precompact neighborhood $U$ of $x$ with $\overline{U} \cap f^{-1}\{f(x)\} = \{x\}$, the degree $\deg(f, f(x), U)$ has the same value. This value is called the \emph{local index} $i(x, f)$ of $f$ at $x$.
	
	In general, the degree $\deg(f, y, U)$ and local index $i(x, f)$ are integer-valued topological invariants that can be defined for continuous $f \colon N \to M$ between oriented topological $n$-manifolds. See e.g.\ the monograph \cite{Fonseca-Gangbo-book} for an exposition on the topological degree and local index for mappings between domains of $\R^n$, as well as how the topological degree connects to continuous Sobolev maps. 
	
	\subsection{Quasiregular values}
	
	The definition of quasiregular values was originally stated in \cite{Kangasniemi-Onninen_1ptReshetnyak}. There, it was given for maps $f \in W^{1, n}_\loc(\Omega, \R^n)$ on an open set $\Omega \subset \R^n$. In particular, such a map has a $(K, \Sigma)$-quasiregular value at $y_0 \in \R^n$ for a given $K \ge 1$ and $\Sigma \in L^1_\loc(\Omega)$ if 
	\[
		\abs{Df(x)}^n \le K J_f(x) + \abs{f(x) - y_0}^n \Sigma(x)
	\]
	for a.e.\ $x \in \Omega$. The version of this definition stated in \eqref{eq:QRval_def} for maps between oriented Riemannian $n$-manifolds is a natural generalization of this definition to the manifold setting.
	
	A key property of quasiregular values that we require is that they satisfy a single-value Reshetnyak's theorem. This is the main result of \cite{Kangasniemi-Onninen_1ptReshetnyak}; we recall the theorem statement for the convenience of the reader.
	
	\begin{thm}[{\cite[Theorem 1.2]{Kangasniemi-Onninen_1ptReshetnyak}}]\label{thm:reshetnyak}
		Let $\Omega \subset \R^n$ be an open domain, and let $f \in C(\Omega, \R^n) \cap W^{1,n}_\loc(\Omega, \R^n)$. Suppose that $f$ has a $(K, \Sigma)$-quasiregular value at $y_0 \in \R^n$, where $K \ge 1$ and $\Sigma \in L^{1+\eps}_\loc(\Omega)$ for some $\eps > 0$. Then either $f$ is the constant function $f \equiv y_0$, or the following conditions hold:
		\begin{enumerate}[label=(\roman*)]
			\item \label{enum:reshetnyak_discr} $f^{-1}\{y_0\}$ is discrete;
			\item \label{enum:reshetnyak_sensepres} for every $x \in f^{-1}\{y_0\}$, we have $i(x, f) > 0$;
			\item \label{enum:reshetnyak_open} for every $x \in f^{-1}\{y_0\}$ and every neighborhood $U \subset \Omega$ of $x$, $y_0$ is in the interior of $fU$.
		\end{enumerate}
	\end{thm}
	
	We also note that one generally assumes that $K \ge 1$ in the definitions of quasiregular mappings and quasiregular values. For quasiregular mappings, this is because every map that is $K$-quasiregular with $0 \le K < 1$ is constant. For quasiregular values, we have the following analogous result shown in \cite{Kangasniemi-Onninen_1ptReshetnyak}.
	
	\begin{thm}[{\cite[Theorem 1.5]{Kangasniemi-Onninen_1ptReshetnyak}}]\label{thm:K_less_than_1}
		Let $\Omega \subset \R^n$ be an open domain, and let $f \in C(\Omega, \R^n) \cap W^{1,n}_\loc(\Omega, \R^n)$. Suppose that $f$ has a $(K, \Sigma)$-quasiregular value at $y_0 \in \R^n$ with $0 \le K < 1$ and $\Sigma \in L^{1+\eps}_\loc(\Omega)$ for some $\eps >0$, then $y_0 \notin f (\Omega)$ or $f \equiv y_0$.
	\end{thm} 

    The above two Theorems from \cite{Kangasniemi-Onninen_1ptReshetnyak} are stated in the Euclidean setting, but we require them for mappings between Riemannian $n$-manifolds. For this reason, we generalize both of them to the Riemannian setting here. This generalization process is relatively straightforward, relying on the following lemma.

    \begin{lemma}\label{lem:QRval_charts}
        Let $N$ and $M$ be oriented, open Riemannian $n$-manifolds, and let $f \in C(N, M) \cap W^{1,n}_\loc(N, M)$ be a map with a $(K, \Sigma)$-quasiregular value at $y_0 \in M$, where $K \ge 0$. Let $U, V \subset \R^n$ be open, and let $\phi \colon U \to N$ and $\psi \colon V \to M$ be smooth, orientation-preserving $L$-bilipschitz charts with $y_0 = \psi(0)$ and $f(\phi U) \subset \psi V$. Then the map $F \colon U \to V$ given by $F = \psi^{-1} \circ f \circ \phi$ has a $(K', \Sigma')$-quasiregular value at $0$ with $K' = L^{4n} K$ and $\Sigma' = L^{3n} (\Sigma \circ \phi)$. Moreover, if, $\Sigma \in L^{p}_\loc(N)$ with $p \in [1, \infty]$, then $\Sigma' \in L^p_\loc(U)$.
    \end{lemma}
    \begin{proof}
        We observe that $F \in C(U, \R^n) \cap W^{1,n}_\loc(U, \R^n)$ by a chain rule of bilipschitz and Sobolev maps. Moreover, we have
        \begin{align*}
            \abs{DF(x)}^n &\le L^{2n} \abs{Df(\phi(x))}^n \\
            &\le L^{2n} K J_f(\phi(x)) + L^{2n} \dist(f(\phi(x)), f(\phi(0)))^n \Sigma(\phi(x))\\
            &\le L^{4n} K J_F(x) + L^{3n} \abs{F(x)}^n \Sigma(\phi(x))\\
            &= K' J_F(x) + \abs{F(x)}^n \Sigma'(x)
        \end{align*}
		for a.e.\ $x \in U$. The final part that $\Sigma'$ has the same local integrability as $\Sigma$ also immediately follows from the fact that $\phi$ is a smooth bilipschitz map.
    \end{proof}

    We first use this lemma to generalize Theorem \ref{thm:reshetnyak}.

    \begin{cor}\label{cor:reshetnyak_mflds}
        Let $N$ and $M$ be connected, oriented, open Riemannian $n$-manifolds, and let $f \in C(N, M) \cap W^{1,n}_\loc(N, M)$ be a map with a $(K, \Sigma)$-quasiregular value at $y_0 \in M$, where $K \ge 1$ and $\Sigma \in L^{1+\eps}_\loc(N)$ for some $\eps > 0$. Then either $f$ is the constant function $f \equiv y_0$, or the following conditions hold:
		\begin{enumerate}[label=(\roman*)]
			\item \label{enum:reshetnyak_discr_mfld} $f^{-1}\{y_0\}$ is discrete;
			\item \label{enum:reshetnyak_sensepres_mfld} for every $x \in f^{-1}\{y_0\}$, $i(x,f)$ is a positive integer;
			\item \label{enum:reshetnyak_open_mfld} for every $x \in f^{-1}\{y_0\}$ and every neighborhood $U \subset \Omega$ of $x$, $y_0$ is in the interior of $fU$.
		\end{enumerate}
    \end{cor}
    \begin{proof}
        Let $x_0 \in f^{-1}\{y_0\}$. We select open, connected chart neighborhoods $\phi U$ and $\psi V$ of $x_0$ and $y_0$, respectively, such that $\phi \colon U \to N$ and $\psi \colon V \to M$ are smooth, orientation-preserving $L$-bilipschitz charts with $f(\phi U) \subset \psi V$, $x_0 = \phi(0)$, and $y_0 = \psi(0)$. By Lemma \ref{lem:QRval_charts}, the map $F \colon U \to V$ defined by $F = \psi^{-1} \circ f \circ \phi$ has a $(K', \Sigma')$-quasiregular value at $0$ with $K' \ge 1$ and $\Sigma' \in L^{1+\eps}_\loc(U)$. 

        By Theorem \ref{thm:reshetnyak} \ref{enum:reshetnyak_discr}, either $F^{-1}\{0\}$ is discrete, or $F^{-1}\{0\} = U$. Thus, either $x_0$ is an isolated point of $f^{-1}\{y_0\}$, or $x_0$ is an interior point of $f^{-1}\{y_0\}$. Since $f^{-1}\{y_0\}$ is closed, and since a closed set remains closed after removing arbitrarily many isolated points from it, we conclude that $\intr (f^{-1}\{y_0\})$ is closed, and hence clopen. Since $N$ is connected, we have either $\intr (f^{-1}\{y_0\}) = N$, in which case $f \equiv y_0$, or $\intr (f^{-1}\{y_0\}) = \emptyset$, in which case every point of $f^{-1}\{y_0\}$ is isolated. Thus, \ref{enum:reshetnyak_discr_mfld} is proven.

        Now, supposing $F^{-1}\{0\}$ is discrete, Theorem \ref{thm:reshetnyak} \ref{enum:reshetnyak_sensepres} yields that $i(0, F)$ is a positive integer. We select a $U_0$ compactly contained in $U$ such that $F^{-1}\{0\} \cap \overline{U_0} = \{0\}$. Thus, $f^{-1}\{y_0\} \cap \overline{\varphi U_0} = \{x_0\}$. Now, if $\omega$ is an $n$-form with unit integral that is supported in the connected component of $y_0$ in $M \setminus f(\partial \phi U_0)$, then $\psi^* \omega$ is an $n$-form with unit integral that is supported in the connected component of $0$ in $V \setminus F(\partial U_0)$, and we may compute
        \[
            \deg(f, y_0, \varphi U_0)
            = \int_{\varphi U_0} f^* \omega
            = \int_{U_0} \phi^* f^* \omega 
            = \int_{U_0} F^* \psi^* \omega
            = \deg(F, 0, U_0).
        \]
        It follows that $i(x_0, f) = i(0, F)$, proving \ref{enum:reshetnyak_sensepres_mfld}.

        Finally, if $U' \subset N$ is a neighborhood of $x$, then $F(\phi^{-1} U')$ is a neighborhood of $0$ by Theorem \ref{thm:reshetnyak} \ref{enum:reshetnyak_open}. Since $\psi$ is a homeomorphism, it follows that $f(U \cap U') = \psi(F(\phi^{-1} U'))$ is a neighborhood $y_0$. Since $f(U \cap U') \subset f(U')$, we conclude that $f(U')$ is a neighborhood of $y_0$, proving \ref{enum:reshetnyak_open_mfld}.
    \end{proof}
	
	We then similarly generalize Theorem \ref{thm:K_less_than_1} to manifolds.
	
	\begin{cor}\label{cor:K_less_than_1_mflds}
		Let $N$ and $M$ be connected, oriented, open Riemannian $n$-manifolds, and let $f \in C(N, M) \cap W^{1,n}_\loc(N, M)$ be a map with a $(K, \Sigma)$-quasiregular value at $y_0 \in M$, where $0 \le K < 1$ and $\Sigma \in L^{1+\eps}_\loc(N)$ for some $\eps>0$. Then $y_0 \notin f(N)$ or $f \equiv y_0$.
	\end{cor}
	\begin{proof}
		Suppose that $x_0 \in f^{-1}\{y_0\}$. We may again select open, connected chart neighborhoods $\phi U$ and $\psi V$ of $x_0$ and $y_0$, respectively, such that $\phi \colon U \to N$ and $\psi \colon V \to M$ are smooth, orientation-preserving $L$-bilipschitz charts with $L$ close enough to $1$ that $L^{4n} K < 1$, and we moreover have $f(\phi U) \subset \psi V$, $x_0 = \phi(0)$, and $y_0 = \psi(0)$. Let $F \colon U \to V$ be given by $F = \psi^{-1} \circ f \circ \phi$. By Lemma \ref{lem:QRval_charts}, $F$ has a $(K', \Sigma')$-quasiregular value at $0 \in \R^n$, where $K' = L^{4n} K < 1$ and $\Sigma' \in L^{1+\eps}_\loc(U)$.
		
		Thus, by Theorem \ref{thm:K_less_than_1}, $F^{-1}\{0\} = U$, and consequently $f^{-1}\{y_0\} \cap \phi U = \phi U$. We have therefore shown that $f^{-1}\{y_0\}$ is open in $N$. Since $N$ is connected, and since $f^{-1}(M \setminus \{y_0\})$ is also open by continuity of $f$, we conclude that one of these sets must equal all of $N$.
	\end{proof}
	
	\subsection{Gehring's lemma}
	
	We then recall a local version of Gehring's lemma that is shown in \cite{Iwaniec_GehringLemma}. As stated in the introduction, this lemma was used in e.g.\ \cite{Kangasniemi-Onninen_1ptReshetnyak} to show higher Sobolev regularity of Euclidean maps with quasiregular values, and we will also use it to show Proposition \ref{prop:higher_int_general}. Note that in the following statement, if $Q \subset \R^n$ is a cube and $\sigma > 0$, then we use $\sigma Q$ to denote the cube $Q$ scaled by $\sigma$ while retaining its center.
	
	\begin{prop}[{\cite[Proposition 6.1]{Iwaniec_GehringLemma}}]\label{prop:local_Gehring_lemma}
		Let $Q_0$ be a cube in $\R^n$, and let $g, h \in L^p(Q_0)$, $1 < p < \infty$, be non-negative functions satisfying
		\[
		\left( \dashint_Q g^p \right)^\frac{1}{p} \leq C_0 \dashint_{2Q} g + \left( \dashint_{2Q} h^p \right)^\frac{1}{p}
		\]
		for all cubes $Q$ with $2Q \subset Q_0$. Then there exists $q_0 = q_0(n, p, C_0) > p$ such that for all $q \in (p, q_0)$ and $\sigma \in (0,1)$, we have
		\[
		\left( \dashint_{\sigma Q_0} g^q \right)^\frac{1}{q} \leq C(n, p, q, \sigma) \left(  \left( \dashint_{Q_0} g^p \right)^\frac{1}{p} + \left( \dashint_{Q_0} h^q \right)^\frac{1}{q} \right).
		\]
	\end{prop}

    \subsection{The Poincar\'e homotopy operator}

    In the following proposition, we recall necessary basic properties of the Poincar\'e homotopy operator of Iwaniec and Lutoborski \cite{Iwaniec-Lutoborski}.

    \begin{prop}[{\cite[Section 4]{Iwaniec-Lutoborski}}]    \label{prop:homotopy_operator}
    	Let $D \subset \R^n$ be a ball with radius $r>0$ or a cube with side length $r>0$. Then there exists a graded linear operator
    	\[
    	    T \colon \bigoplus_{k=1}^n L^1_\loc(\wedge^k T^*D) 
            \to \bigoplus_{k=1}^n L^1_\loc(\wedge^{k-1} T^*D)
    	\]
    	satisfying the following conditions for all $k \in \{1, \dots, n\}$:
    	\begin{enumerate} [label=(\roman*)]
    		\item \label{item:homotopy_continuity} For each $p \in [1, \infty]$, and each $\omega \in L^p(\wedge^k T^*D)$, we have
    		\[
    		      \norm{T \omega}_{L^{p}(D)} 
                \le C(n)r \norm{\omega}_{L^{p}(D)};
    		\]
    		\item \label{item:homotopy_compactness} For each $p \in [1, \infty]$, the operator $T$ is compact from $L^p(\wedge^k T^* D)$ to $L^p(\wedge^{k-1} T^* D)$;
    		\item \label{item:homotopy} If $p \in [1, \infty]$ and $\omega \in W^{d,p}(\wedge^k T^* D)$, then $T \omega \in W^{d,p}(\wedge^{k-1} T^* D)$ and $\omega = Td\omega + dT\omega$;
    		\item \label{item:homotopy_embedding}
            If $p \in [1, \infty]$, $q \in (1, \infty]$ are such that $q^{-1} \le p^{-1} + n^{-1}$ and $(p, q) \ne (\infty, n)$, 
            and if $\omega \in L^q(\wedge^k T^* D)$,
            then
    		\[
    		\norm{T\omega}_{L^p(D)} \le C(n, p, q) r^{\left(\frac{n}{p} + 1 - \frac{n}{q}\right)} \norm{\omega}_{L^q(D)}.
    		\]
    	\end{enumerate}
    \end{prop}
    \begin{proof}
        Part \ref{item:homotopy_continuity} is \cite[(4.15)]{Iwaniec-Lutoborski}. For part \ref{item:homotopy_compactness}, see the discussion in Appendix \ref{sect:Poincare_homotopy_compactness}.
        Part \ref{item:homotopy} is \cite[Lemma 4.2]{Iwaniec-Lutoborski}. The statement of part \ref{item:homotopy} in \cite{Iwaniec-Lutoborski} excludes $p = \infty$, but if $\omega \in W^{d,\infty}(\wedge^k T^* D)$, then $\omega \in W^{d,1}(\wedge^k T^* D)$ since $D$ is bounded, we get $\omega = Td\omega + dT\omega$ a.e.\ by the case $p = 1$ of \ref{item:homotopy}, and it follows that $\norm{d T\omega}_{L^\infty(D)} \le \norm{Td\omega}_{L^\infty(D)} + \norm{\omega}_{L^\infty(D)} < \infty$ by \ref{item:homotopy_continuity}.

        For part \ref{item:homotopy_embedding}, case $q = \infty$ follows since 
        \[
            \norm{T\omega}_{L^p(D)} 
            \le C(n,p) r^{\frac{n}{p}}\norm{T\omega}_{L^\infty(D)} 
            \le C(n,p) r^{\frac{n}{p} + 1}\norm{\omega}_{L^\infty(D)}
        \]
        by H\"older's inequality and part \ref{item:homotopy_continuity}. If on the other hand $1 < q < \infty$, then \cite[Proposition 4.1]{Iwaniec-Lutoborski} yields that the coefficient functions $T_I\omega$ of $T\omega$ with respect to the standard basis of $\wedge^k (\R^n)^*$ are in $W^{1,q}(D)$, and satisfy 
        \[
            r^{-1}\norm{T\omega}_{L^q(D)} + \sum_I \norm{\nabla T_I \omega}_{L^q(D)} \le C(n, q) \norm{\omega}_{L^q(D)}.
        \]
        Note here the factor $(\diam D)^{-1}$ in the definition of the $\norm{\cdot}_{W^{1,p}(D)}$-norm in \cite[Chapter 3]{Iwaniec-Lutoborski}, which produces the $r^{-1}$ above. Since $q^{-1} \le p^{-1} + n^{-1}$ and $(p,q) \ne (\infty, n)$, the Sobolev-Poincar\'e inequality yields 
        \[
            \norm{T_I\omega - (T_I \omega)_D}_{L^p(D)} \le C(n,p,q) r^{\frac{n}{p} + 1 - \frac{n}{q}} \norm{\nabla T_I \omega}_{L^q(D)}.
        \]
        Since we have $\norm{T_I\omega}_{L^p(D)} \le \norm{T_I\omega - (T_I \omega)_D}_{L^p(D)} + \norm{(T_I \omega)_D}_{L^p(D)}$ and $\norm{(T_I \omega)_D}_{L^p(D)} \le C(n,p,q) r^{n/p - n/q} \norm{T_I \omega}_{L^q(D)}$, the claim follows.
        
    \end{proof}

    Moreover, for every $r > 0$, we use $T_r$ to denote a fixed Poincar\'e homotopy operator specifically on the ball $\B^n_r$.

    \subsection{Weak and vague convergence of $k$-forms}
    Given $p \in [1, \infty]$, and an open $\Omega \subset \R^n$, we say that a sequence $\omega_j \in L^p(\wedge^k T^* \Omega)$ converges \emph{$L^p$-weakly} to $\omega \in L^p(\wedge^k T^* \Omega)$, denoted $\omega_j \weakto \omega$, if
    \begin{equation}\label{eq:weak_vague_conv}
        \lim_{j \to \infty} \int_\Omega \omega_j \wedge \eta = \int_{\Omega} \omega \wedge \eta
    \end{equation}
    for all $\eta \in L^{p^*}(\wedge^{n-k} T^* \Omega)$. A careful reader notes that this is a slight abuse of terminology, as for instance for $p = \infty$, this convergence is technically weak*-convergence. We also say that a sequence $\omega_j \in L^1_\loc(\wedge^k T^* \Omega)$ converges \emph{vaguely} to $\omega \in L^1_\loc(\wedge^k T^* \Omega)$, also denoted $\omega_j \weakto \omega$, if \eqref{eq:weak_vague_conv} holds for all $\eta \in C_0(\wedge^{n-k} T^* \Omega)$. This terminology is derived from the standard vague convergence of measures. Note that $L^p$-weak convergence always implies vague convergence.
        
    We recall two basic facts about these convergence types. First, under suitable circumstances, we can check these types of convergence with $C^\infty_0$-regular test forms. The proof is an immediate consequence of $C^\infty_0(\wedge^{n-k} T^* \Omega)$ being dense in $L^q(\wedge^{n-k} T^* \Omega)$ if $q \in [1, \infty)$, and $C^\infty_0(\wedge^{n-k} T^* \Omega)$ being dense in $C_0(\wedge^{n-k} T^* \Omega)$ under the $L^\infty$-norm.
        
    \begin{lemma}\label{lem:smooth_to_weak_or_vague_conv}
    	Let $\Omega \subset \R^n$ be open, let $p \in [1, \infty]$, let $\omega, \omega_j \in L^p(\wedge^k T^* \Omega)$ for $j \in \Z_{>0}$, and suppose that $(\norm{\omega_j}_{L^p(\Omega)})$ is bounded. If
       	\[
           	\lim_{j \to \infty} \int_\Omega \omega_j \wedge \eta = \int_\Omega \omega \wedge \eta
        \]
       	for all $\omega \in C^\infty_0(\wedge^{n-k} T^* \Omega)$, then $\omega_j \weakto \omega$ in the $L^p$-weak sense if $p > 1$, and $\omega_j \weakto \omega$ vaguely if $p = 1$.
    \end{lemma}

    The other key fact we recall is that $L^1$-integrable vague limits, and consequently also $L^p$-weak limits, are unique. We provide the proof for the convenience of the reader.
        
    \begin{lemma}\label{lem:vague_limits_unique}
    	Let $\Omega \subset \R^n$ be open, let $\omega_j \in L^1_\loc(\wedge^k T^* \Omega)$ for $j \in \Z_{>0}$, and let $\omega, \omega' \in L^1_\loc(\wedge^k T^* \Omega)$ be vague limits of $(\omega_j)$. Then $\omega = \omega'$ a.e.\ in $\Omega$. 
    \end{lemma}
    \begin{proof}
        By vague convergence, we have
        \[
        	\int_\Omega (\omega - \omega') \wedge \eta = \lim_{j \to \infty} \int_\Omega (\omega_j - \omega_j) \wedge \eta = 0
        \]
        for all $\eta \in C_0(\wedge^{n-k} T^* \Omega)$. Now, suppose that $x \in \Omega$ is a Lebesgue point of $\omega - \omega'$. For every $r > 0$ such that $\overline{\B}^n(x, r) \subset \Omega$, we use $\mathcal{X}_{r}$ to denote the characteristic function of $\overline{\B}^n(x, r)$. For every unit $(n-k)$-covector $\eps_I \in \wedge^{n-k} (\R^n)^*$, we approximate $\mathcal{X}_{r} \eps_I$ with functions $\eta_{r, j} \eps_I$ such that $\eta_{r, j} \le 1$, $\eta_{r, j} \equiv 1$ on $\overline{\B}^n(x, r)$, and $\spt (\eta_{r,j}) \subset \B^n(x, r + j^{-1})$. 
        	
        Then, noting that $\B^n(x, r + j^{-1}) \subset \Omega$ for large enough $j$, we have
        \begin{multline*}
        	\limsup_{j \to \infty} \abs{ \int_\Omega (\omega - \omega') 
        		\wedge (\mathcal{X}_{r} - \eta_{r,j}) \eps_I}\\ 
        	\le \limsup_{j \to \infty} 
                \norm{\omega - \omega'}_{L^1(\B^n(x, r + j^{-1}) 
                    \setminus \B^n(x, r))}
        	= 0
        	\end{multline*}
        since $\omega - \omega' \in L^1_\loc(\Omega)$. Thus, since $\eta_{r, j} \eps_I \in C_0(\wedge^{n-k} T^* \Omega)$ for large $j$, we obtain that
        \[
        	\int_{\B^n(x, r)} (\omega - \omega') \wedge \eps_I = 0.
        \]
        Since this holds for all sufficiently small $r > 0$ and all unit $\eps_I \in \wedge^{n-k} (\R^n)^*$, it follows that $(\omega - \omega')_x = 0$ by the Lebesgue differentiation theorem. And as this holds for a.e.\ $x \in \Omega$, the claim follows.
    \end{proof}

    \subsection{The K\"unneth ideal}

    In general, a closed form whose de Rham class is in the K\"unneth ideal of a smooth manifold can be expressed as a sum of products of closed forms up to an additive exact term. The following result states that the additive exact term can be omitted on a closed, oriented Riemannian manifold.

    \begin{lemma} \label{lem:kunneth}
        Let $m\ge n\ge 2$ and let $M$ be a closed, oriented Riemannian $m$-manifold. If $\omega \in C^\infty(\wedge^n T^* M)$ is a closed $n$-form on $M$ whose de Rham class $[\omega]$ is in the K\"unneth ideal $K^*(M)$, then there exist $k_1,\ldots,k_j \in \{1,\ldots,n-1\}$, closed forms $\alpha_1,\ldots,\alpha_j \in C^\infty(\wedge^* T^* M)$, and closed forms $\beta_1,\ldots,\beta_j \in C^\infty(\wedge^* T^* M)$ satisfying $\omega = \sum_{i=1}^j \alpha_i \wedge \beta_i$ and $\alpha_i \in C^\infty(\wedge^{k_i} T^* M)$ for $i=1,\ldots,j$.
    \end{lemma}

    \begin{proof}
        Since the de Rham class $[\omega]$ belongs to the K\"unneth ideal $K^*(M)$, there exist $\tau \in C^\infty(\wedge^{n-1} T^* M)$, $k_1,\ldots,k_l \in \{1,\ldots,n-1\}$, closed forms $\alpha_1,\ldots,\alpha_l \in C^\infty(\wedge^* T^* M)$, and closed forms $\beta_1,\ldots,\beta_l \in C^\infty(\wedge^* T^* M)$ satisfying $\omega = \left( \sum_{i=1}^l \alpha_i \wedge \beta_i \right) + d\tau$ and $\alpha_i \in C^\infty(\wedge^{k_i} T^* M)$ for $i=1,\ldots,l$. On the other hand, by \cite[Proposition 2.8]{Hajlasz-Iwaniec-Maly-Onninen}, there exist $k_{l+1},\ldots,k_j \in \{1,\ldots,n-1\}$, closed forms $\alpha_{l+1},\ldots,\alpha_j \in C^\infty(\wedge^* T^* M)$, and closed forms $\beta_{l+1},\ldots,\beta_j \in C^\infty(\wedge^* T^* M)$ for which $d\tau=\sum_{i=l+1}^j \alpha_i \wedge \beta_i$ and $\alpha_i \in C^\infty(\wedge^{k_i} T^* M)$ for $i=l+1,\ldots,j$. This concludes the proof.
    \end{proof}

    \section{Proof of Theorem \ref{thm:QRcurve_with_sigma_embedding}}
    \label{sect:proof_of_cohom_embedding}

    In this section, we provide the main new content of the proof of Theorem \ref{thm:QRcurve_with_sigma_embedding}. The proof of a single result, Proposition \ref{prop:technical_abstraction}, is postponed to the end of the paper.

    \subsection{Abstraction of the main argument of \cite{Heikkila-Pankka_Elliptic} and \cite{Heikkila_Embedding-curves}}\label{subsect:abstraction}

    The core part of our strategy for proving Theorem \ref{thm:QRcurve_with_sigma_embedding} closely follows that of \cite{Heikkila-Pankka_Elliptic} and \cite{Heikkila_Embedding-curves}. If we were to write the argument similarly as in these prior works, a large portion of the proof would end up being repetition of the existing argument, where the changes caused by replacing \eqref{eq:QRcurve_def} with \eqref{eq:QR_curve_with_Sigma} would not play any role. 
    
    For this reason, we proceed to state a technical proposition which isolates the abstract $L^p$-theoretical core of the argument from \cite{Heikkila-Pankka_Elliptic} and \cite{Heikkila_Embedding-curves}. The result is stated in relatively high generality, in hopes of reducing the need for these steps to be repeated in any future work.

    We begin by fixing some terminology. Let $B \subset \R^n$ be an open ball, and let $M$ be a closed, connected, oriented Riemannian $m$-manifold, where $m \ge n \ge 2$. We use $1_M$ and $1_B$ to denote the map $x \mapsto 1$ on $M$ and $B$, respectively.
    
    We say that a sequence of exponents $p_0, p_1, \dots, p_n \in [1, \infty]$ is a \emph{H\"older sequence} if $p_k^{-1} + p_l^{-1} \le p_{k+l}^{-1}$ for all $k, l \in \{0, \dots, n\}$ satisfying $k+l\le n$. We note that this property implies $p_{k+1} \le p_k$ for all $k \in \{0, \dots, n-1\}$, $p_k^* \le p_{n-k}$ for all $k \in \{0, \dots, n\}$, and $p_0 = \infty$. 
    Our reason for this definition is that if $p_0, \dots, p_n \in [1, \infty]$ is a H\"older sequence, then
    \begin{equation}\label{eq:Holder_seq_est}
    	\norm{\varphi \psi}_{L^{p_{k+l}}(B)} 
    		\le C(B) \norm{\varphi}_{L^{p_{k}}(B)} \norm{\psi}_{L^{p_{l}}(B)}
    \end{equation}
    for all $\varphi, \psi \in L^1_\loc(B)$, which in turn implies that
    \[
    	L^{p_0, \dots, p_n}(\wedge^* T^* B) := \bigoplus_{k=0}^n L^{p_k}(\wedge^k T^* B)
    \]
    is a graded unital algebra when equipped with the wedge product and unit $1_B$. 
    
    Similarly, for a H\"older sequence $p_0, \dots, p_n \in [1, \infty]$, we define
    \[
        L^{p_0, \dots, p_n}(\wedge^* T^* B)  \cap \ker(d) := \bigoplus_{k=0}^n L^{p_k}(\wedge^k T^* B) \cap \ker(d),
    \]
    where we use $L^{p_k}(\wedge^k T^* B) \cap \ker(d)$ to denote the space of weakly closed $k$-forms in $L^{p_k}(\wedge^k T^* B)$ for all $k \in \{0, \dots, n\}$. We observe that the space $L^{p_0, \dots, p_n}(\wedge^* T^* B)  \cap \ker(d)$ is closed under the wedge product, and is hence a graded unital subalgebra of $L^{p_0, \dots, p_n}(\wedge^* T^* B)$.
    
    Then, given a H\"older sequence $p_0, \dots, p_n \in [1, \infty]$, we define that a map $G \colon C^\infty(\wedge^* T^* M) \to L^{p_0, \dots, p_n}(\wedge^* T^* B)$ is a \emph{generalized pull-back} if $G$ is a graded linear map satisfying $G(1_M) = 1_B$, $G(\alpha \wedge \beta) = G(\alpha) \wedge G(\beta)$, and $G(d\alpha) = dG(\alpha)$ in the weak sense for all $\alpha, \beta \in C^\infty(\wedge^* T^* M)$. 
    
    With these definitions, we are ready to state our technical abstraction of the argument from \cite{Heikkila-Pankka_Elliptic} and \cite{Heikkila_Embedding-curves}.

    \begin{prop}\label{prop:technical_abstraction}
    	Let $B \subset \R^n$ be an open ball, and let $M$ be a closed, connected, oriented Riemannian $m$-manifold, where $m \ge n \ge 2$. Suppose that $p_0, \dots, p_{n-1} \in (1, \infty]$, $p_n \in [1, \infty]$ is a H\"older sequence of exponents.
        \begin{enumerate}[label=(\roman*)]
            \item \label{enum:cohom_map_into_forms}
            Suppose that $G_j \colon C^\infty(\wedge^* T^* M) \to L^{p_0, \dots, p_n}(\wedge^* T^* B)$ are generalized pull-backs for all $j \in \Z_{> 0}$, $A_j \in (0, \infty)$ with $\lim_{j \to \infty} A_j = \infty$, and $C \in (0, \infty)$ is a constant for which
            \begin{equation}\label{eq:Aj_condition}
    		      \norm{G_j(\alpha)}_{L^{p_k}(B)} \le C A_j^\frac{k}{n} \norm{\alpha}_{L^\infty(M)}
    	    \end{equation}
    	    for all $k \in \{0, \dots, n\}$, $j \in \Z_{> 0}$, and $\alpha \in C^\infty(\wedge^k T^* M)$. Then there exists a subsequence $(G_{j_i})$ of $(G_j)$ and a graded algebra homomorphism 
    	    \[
    		      L \colon H^*_\derham(M) 
                \to L^{p_0, \dots, p_n}(\wedge^* T^* B) \cap \ker(d)
    	    \]
    	    such that for all $k \in \{0, \dots, n\}$, $c \in H^k_\derham(M)$, and $\omega \in c$, if $p_k > 1$ or if $c$ is in the $n$:th layer $K^n(M)$ of the K\"unneth ideal of $M$, then
            \[
    		      A_{j_i}^{-k/n} G_{j_i} (\omega) \weakto L(c).
    	    \]
    	    Here, the above convergence is $L^{p_k}$-weak convergence if $p_k > 1$ and vague convergence if $p_k = 1$.
            \item \label{enum:cohom_map_into_ext_alg}
            If $L \colon H^*_\derham(M) \to L^{p_0, \dots, p_n}(\wedge^* T^* B) \cap \ker(d)$ is a graded algebra homomorphism, and if $c \in H^k_\derham(M)$, $k \in \{0, \dots, n\}$, is such that $L(c)$ is not a.e.\ vanishing, then there exists a graded homomorphism of algebras $\Phi \colon H^*_\derham(M) \to \wedge^* \R^n$ for which $\Phi(c) \ne 0$. Moreover, if $k = n = m$, then $\Phi$ is injective.
        \end{enumerate}
    \end{prop}

    We postpone the proof of Proposition \ref{prop:technical_abstraction} to Section \ref{sect:abstraction_proof} at the end of the article.

    We then outline how the prior known embedding results follow from Proposition \ref{prop:technical_abstraction}, and how the argument differs in our setting. Notably, the statement of Proposition \ref{prop:technical_abstraction} is general enough that it also implies the corresponding embedding result for Lipschitz maps of positive asymptotic degree in the equidimensional case; see \cite[Theorem 2.3]{Berdnikov-Guth-Manin} or \cite[Section 2]{Manin-Prywes}. As this is the simplest known case to demonstrate the use of Proposition \ref{prop:technical_abstraction}, we recall the statement and sketch a proof using the proposition.

    \begin{thm}[{Special case of \cite[Theorem 2.3]{Berdnikov-Guth-Manin}}]\label{thm:Berdnikov-Guth-Manin}
        Let $M$ be a closed, connected, oriented Riemannian $n$-manifold. Suppose that there exists an $L$-lipschitz map $f \colon \R^n \to M$ with
        \[
            \limsup_{r \to \infty} \frac{1}{r^n} \int_{\B_r^n} f^* \vol_M > 0.
        \]
        Then there exists a graded embedding of algebras $\Phi \colon H^*_\derham(M) \hookrightarrow \wedge^* \R^n$
    \end{thm}
    \begin{proof}[Sketch of proof using Proposition \ref{prop:technical_abstraction}]
        For $j\in \Z_{>0}$, let $r_j \in (0, \infty)$ be radii such that $\lim_{j \to \infty} r_j = \infty$ and
        \begin{equation} \label{eq:lip_positive_asympt_degree_seq}
            \inf_{j \in \Z_{>0}} \frac{1}{r_j^n} \int_{\B^n_{r_j}} f^* \vol_M > 0,
        \end{equation}
        and let $f_j = f \circ s_j$, where $s_j \colon \B^n \to \B^n_{r_j}$ is the scaling map $x \mapsto r_j x$. Now every $f_j$ is $(r_j L$)-Lipschitz, and hence for every $\omega \in C^\infty(\wedge^k T^* M)$, we have $f_j^* \omega \in W^{d, \infty}(\wedge^k T^* \B^n)$ with $\smallnorm{f_j^* \omega}_{L^\infty(\B^n)} \le (L r_j)^k \norm{\omega}_{L^\infty(M)}$ and $f_j^* d\omega = d f_j^* \omega$ weakly. 
        
        In particular, Proposition \ref{prop:technical_abstraction} \ref{enum:cohom_map_into_forms} applies with $p_0 = \ldots = p_n = \infty$, $G_j = f_j^*$, $A_j = r_j^n$, and $C = L^n$. Thus, we find a graded algebra homomorphism $L \colon H^*_\derham(M) \to L^{p_0, \dots, p_n}(\wedge^* T^* \B^n) \cap \ker(d)$ with $r_j^{-n} f_{j_i}^* \vol_M \weakto L([\vol_M])$, where the convergence is $L^\infty$-weak. Now, the constant function $x \mapsto 1$ on $\B^n$ is an admissible test function for $L^\infty$-weak convergence of $n$-forms, so \eqref{eq:lip_positive_asympt_degree_seq} implies that $L([\vol_M])$ has positive integral, allowing the application of Proposition \ref{prop:technical_abstraction} \ref{enum:cohom_map_into_ext_alg} to complete the proof.
    \end{proof}

    Next, we outline the proof of the embedding theorem for quasiregular $\omega$-curves $F \colon \R^n \to M$, which generalizes the result for quasiregular maps. The strategy in this case resembles the proof of Theorem \ref{thm:Berdnikov-Guth-Manin}, but with some added complications. Suppose that it is already known that $\norm{DF}_{L^n(\R^n)} = \infty$. One again considers maps $F_j = F \circ s_j$, where $s_j \colon \B^n \to B_j$ map $\B^n$ to increasingly large balls $B_j \subset \R^n$. One may then apply Proposition \ref{prop:technical_abstraction} \ref{enum:cohom_map_into_forms} with $p_k = n/k$, $G_j = F_j^*$, $C$ depending on the distortion constant $K$ and $\omega$, and
    \[
        A_j = \int_{B_j} \hodge F^* \omega.
    \]
    Here, $\norm{DF}_{L^n(\R^n)} = \infty$ and \eqref{eq:QRcurve_def} are used to ensure that $A_j \to \infty$. However, the complication that arises is that the convergence $A_j^{-1} F_j^* \omega \weakto L([\omega])$ is only vague, which does not allow the use of $x \mapsto 1$ as a test function. To work around this, the balls $B_j$ are selected specifically to ensure a doubling property which ensures that $A_j^{-1} F_j^* \omega$ cannot tend vaguely to zero.

    We then reach the setting of our current article, where $F$ satisfies \eqref{eq:QR_curve_with_Sigma}. The main difference compared to quasiregular $\omega$-curves is that we need to find a new normalizing factor $A_j$ in order to apply Proposition \ref{prop:technical_abstraction}. It turns out that the right hand side of \eqref{eq:QR_curve_with_Sigma} yields a relatively natural viable normalizing factor, 
    \[
        A_j = \int_{B_j} (K\hodge F^* \omega + \Sigma).
    \]
    Thus, the main new part of the proof is verifying that this new normalizing factor still both yields \eqref{eq:Aj_condition} and ensures that $A_j^{-1} F_j^* \omega$ cannot converge vaguely to 0.

    \subsection{Localization of Theorem \ref{thm:QRcurve_with_sigma_embedding}}

    We then proceed to prove Theorem \ref{thm:QRcurve_with_sigma_embedding} using Proposition \ref{prop:technical_abstraction}. Our first objective is to use rescaling to reduce the global Theorem \ref{thm:QRcurve_with_sigma_embedding} to a local problem on a single ball. We begin by recalling Rickman's Hunting Lemma \cite[Lemma 5.1]{Rickman_Picard}. The following formulation is due to Bonk and Poggi-Corradini \cite[Lemma 2.1]{Bonk_PoggiCorradini-Rickman_Picard}.

    \begin{lemma} \label{lem:Rickman-Hunting}
        Let $\mu$ be an atomless Borel measure on $\R^n$ satisfying $\mu(\R^n)=\infty$ and $\mu(\B)<\infty$ for every open ball $\B \subset \R^n$. Then there exists a constant $D=D(n)>1$ with the property that, for every $j\in \Z_{> 0}$, there exists an open ball $\B \subset \R^n$ for which
        \[
            j\le \mu(2\B)\le D\mu(\B).
        \]
    \end{lemma}
    
    Next, let $m \ge n \ge 2$, let $M$ be a closed, connected, oriented Riemannian $m$-manifold, and let $\omega \in C^\infty(\wedge^n T^* M)$ be a closed non-vanishing $n$-form on $M$ for which the de Rham class $[\omega]$ is in the K\"unneth ideal $K^*(M)$. We also let $K\ge 1$ and $D\ge 1$ be constants. 
    
    We use $\ff$ to denote the family of mappings $F\in C(\B^n_2,M)\cap W^{1,n}(\B^n_2,M)$ for which there exists a $\Sigma \in L^1(\B^n_2)$ with the following properties:
        \begin{enumerate} [label=(\roman*)]
            \item the map $F$ satisfies \eqref{eq:QR_curve_with_Sigma} with $K$, $\Sigma$, and $\omega$;
            \item \[
            0 < \int_{\B^n_2} \left( K \hodge F^* \omega + \Sigma \right) \le D \int_{\B^n} \left( K \hodge F^* \omega + \Sigma \right);
            \]
            \item \[
            \inf_{x\in \B^n_2} \abs{\omega_{F(x)}}_{\mass} \ge D^{-1};
            \]
            \item \[
            \int_{\B^n_2} \abs{\Sigma} \le D.
            \]
        \end{enumerate}
        We say that such a $\Sigma \in L^1(\B^n_2)$ is \emph{$F$-compatible}. Moreover, for all $F\in \ff$ and all $F$-compatible $\Sigma \in L^1(\B^n_2)$, we denote
        \[
        A(F,\Sigma) = \int_{\B^n} \left( K \hodge F^* \omega + \Sigma \right).
        \]
        We also denote $A(F)=\sup_{\Sigma} A(F,\Sigma)$ for $F\in \ff$, where the supremum is taken over all $F$-compatible $\Sigma \in L^1(\B^n_2)$. Note that for every $F \in \ff$, we have
        \[
            A(F) 
            \le K \norm{\omega}_{L^\infty(M)}\norm{D F}_{L^n(\B^n)}^n + D 
            < \infty.
        \]

        We then state the localized version of Theorem \ref{thm:QRcurve_with_sigma_embedding}, which is formulated in terms of properties of $\ff$. The corresponding results in the settings of quasiregular maps and quasiregular curves are \cite[Theorem 1.3]{Heikkila-Pankka_Elliptic} and \cite[Theorem 1.10]{Heikkila_Embedding-curves}, respectively.

        \begin{thm} \label{thm:local_version_of_QRcurve_with_sigma_embedding}
            Let $m \ge n \ge 2$, let $M$ be a closed, connected, oriented Riemannian $m$-manifold, let $\omega \in C^\infty(\wedge^n T^* M)$ be a closed non-vanishing $n$-form on $M$ for which the de Rham class $[\omega]$ is in the K\"unneth ideal $K^*(M)$, let $K\ge 1$, and let $D\ge 1$. If
            \[
                \sup_{F\in \ff} A(F)=\infty,
            \]
            then there exists a graded homomorphism of algebras 
		      \[
			    \iota \colon H_\derham^*(M) \to \wedge^* \R^n
		      \]
		      for which $\iota([\omega]) \ne 0$.
        \end{thm}

        We then proceed to show how Theorem \ref{thm:local_version_of_QRcurve_with_sigma_embedding} implies Theorem \ref{thm:QRcurve_with_sigma_embedding}. The essence of this proof is contained in the following proposition; see also \cite[Proposition 2.1]{Heikkila_Embedding-curves}, \cite[Proposition 2.1]{Heikkila-Pankka_Elliptic}, and \cite[Section 4]{Prywes_Annals}.

        \begin{prop} \label{prop:reduction}
            Let $m \ge n \ge 2$, let $M$ be a closed, connected, oriented Riemannian $m$-manifold, let $\omega \in C^\infty(\wedge^n T^* M)$ be a closed non-vanishing $n$-form on $M$ for which the de Rham class $[\omega]$ is in the K\"unneth ideal $K^*(M)$, let $K \ge 1$, and let $\Sigma \in L^1(\R^n)$. If there exists a map $F \in C(\R^n, M) \cap W^{1,n}_\loc(\R^n, M)$ for which $\norm{DF}_{L^n(\R^n)} = \infty$ and $F$ satisfies \eqref{eq:QR_curve_with_Sigma} with $K$, $\Sigma$, and $\omega$, then there exists a constant $D=D(\omega,\Sigma)>1$ for which
            \[
            \sup_{F'\in \ff} A(F') = \infty.
            \]
        \end{prop}

        \begin{proof}
            The measure $\mu(E)=\int_E \left( K \hodge F^* \omega + \Sigma \right)$ satisfies the assumptions of Lemma \ref{lem:Rickman-Hunting}; indeed, \eqref{eq:QR_curve_with_Sigma} ensures that $\mu$ has non-negative values, the local integrability of $\abs{DF}^n$ and $\Sigma$ ensures that $\mu$ is atomless and finite on balls, and $\mu(\R^n) = \infty$ follows from \eqref{eq:QR_curve_with_Sigma} and $\norm{DF}_{L^n(\R^n)} = \infty$. Hence, there exists a constant $D=D(n)>1$ and a sequence of open balls $(\B^n(a_j,r_j))_{j\in \Z_{> 0}}$ for which
            \[
            j \le \int_{\B^n(a_j,2r_j)} \left( K \hodge F^* \omega + \Sigma \right) \le D \int_{\B^n(a_j,r_j)} \left( K \hodge F^* \omega + \Sigma \right).
            \]
            Let $D'=D'(\omega,\Sigma)\ge D$ satisfy $\inf_{x\in M} \abs{\omega_x}_{\mass} \ge (D')^{-1}$ and $\norm{\Sigma}_{L^1(\R^n)} \le D'$. Then, for each $j\in \Z_{> 0}$, the map $F_j \colon \B^n_2 \to M$, $x\mapsto F(r_jx+a_j)$, belongs to $\mathcal{F}_{K,D'}(M,\omega)$; indeed, $\Sigma_j\in L^1(\B^n_2)$ defined by $x\mapsto r^n_j \Sigma(r_jx+a_j)$ is $F_j$-compatible. Moreover, by our selection condition for the balls $\B^n(a_j,r_j)$, we have $A(F_j)\ge A(F_j,\Sigma_j) \ge j/D$ for all $j \in \Z_{> 0}$. Thus, the claim holds with the choice of constant $D'$.
        \end{proof}

        We then complete the process of reducing Theorem \ref{thm:QRcurve_with_sigma_embedding} to Theorem \ref{thm:local_version_of_QRcurve_with_sigma_embedding}
        
        \begin{proof}[Proof of Theorem \ref{thm:QRcurve_with_sigma_embedding} assuming Theorem \ref{thm:local_version_of_QRcurve_with_sigma_embedding}]
            Theorem \ref{thm:local_version_of_QRcurve_with_sigma_embedding} and Proposition \ref{prop:reduction} immediately imply the claim in the case $K \ge 1$. Hence, it remains to consider the case $0 \le K < 1$. We show that in this case, there are no maps $F$ which satisfy the given assumptions, and the statement is hence trivially true. 
            
            Indeed, suppose towards contradiction that $F$ satisfies the given conditions with $0 \le K < 1$. Then, by \eqref{eq:QR_curve_with_Sigma} and the point-wise estimate $\hodge (F^* \omega)_x \le \abs{\omega_{F(x)}}_{\mass} \abs{DF(x)}^n$ for a.e.\ $x \in \R^n$, we get $\abs{\omega_{F(x)}}_{\mass} \abs{DF(x)}^n \le (1-K)^{-1} \Sigma(x)$ for a.e.\ $x \in \R^n$. Notably, since $M$ is compact, and since $\omega$ is continuous and non-vanishing, $\abs{\omega}_{\mass}$ has a positive minimum on $M$, and therefore 
            \[
                \norm{DF}_{L^n(\R^n)}^n \le \left((1-K)\min_{x \in M} \abs{\omega_x}_{\mass}\right)^{-1} \norm{\Sigma}_{L^1(\R^n)} < \infty.
            \]
            This contradicts $\norm{DF}_{L^n(\R^n)} = \infty$; thus no maps $F$ satisfying the given conditions exist if $0 \le K < 1$.
        \end{proof}
        
        \subsection{Proof of Theorem \ref{thm:local_version_of_QRcurve_with_sigma_embedding}}
        To complete the proof of Theorem \ref{thm:QRcurve_with_sigma_embedding}, it hence remains to prove Theorem \ref{thm:local_version_of_QRcurve_with_sigma_embedding}. For this, we use Proposition \ref{prop:technical_abstraction}. In order to apply this proposition, we first need the following lemma.

        \begin{lemma} \label{lem:normalized_is_bounded}
            Let $m \ge n \ge 2$, let $M$ be a closed, connected, oriented Riemannian $m$-manifold, let $\omega \in C^\infty(\wedge^n T^* M)$ be a closed non-vanishing $n$-form on $M$ for which the de Rham class $[\omega]$ is in the K\"unneth ideal $K^*(M)$, let $K\ge 1$, and let $D\ge 1$. If $F\in \ff$, then we have
            \[
                \norm{F^* \alpha}_{L^{n/k}(\B_2^n)} \le D^2 
                A^\frac{k}{n}(F) \norm{\alpha}_{L^\infty(M)}
            \]
            for $\alpha \in C^\infty(\wedge^k T^* M)$ and $k=0,\ldots,n$.
        \end{lemma}

        \begin{proof}
            The claim holds trivially if $k=0$ as $\abs{F^*(\alpha)}(x)=\abs{\alpha(F(x))}$ for each $x\in \B_2^n$. Hence, suppose that $0<k\le n$ and let $\Sigma \in L^1(\B_2^n)$ be $F$-compatible. Then
            \[
                \norm{F^* \alpha}_{L^{n/k}(\B_2^n)} = \left( \int_{\B_2^n} \abs{F^* \alpha}^\frac{n}{k} \right)^\frac{k}{n}
                \le \norm{\alpha}_{L^\infty(M)} \left( \int_{\B_2^n} \abs{DF}^n \right)^\frac{k}{n}.
            \]
            Since $F\in \ff$, we may estimate
            \begin{align*}
                \int_{\B_2^n} \abs{DF}^n &\le D \int_{\B_2^n} (\abs{\omega}_{\mass} \circ F) \abs{DF}^n \le D \int_{\B_2^n} \left( K \hodge F^* \omega + \Sigma \right) \\
                &\le D^2 \int_{\B^n} \left( K \hodge F^* \omega + \Sigma \right) = D^2 A(F,\Sigma) \le D^2 A(F).
            \end{align*}
            Thus
            \[
                \norm{F^* \alpha}_{L^{n/k}(\B_2^n)} 
                \le (D^2 A(F))^\frac{k}{n} \norm{\alpha}_{L^\infty(M)} 
                \le D^2 A^\frac{k}{n}(F) \norm{\alpha}_{L^\infty(M)},
            \]
            which concludes the proof.
        \end{proof}

        We are now ready to prove Theorem \ref{thm:local_version_of_QRcurve_with_sigma_embedding}.

        \begin{proof}[Proof of Theorem \ref{thm:local_version_of_QRcurve_with_sigma_embedding} assuming Proposition \ref{prop:technical_abstraction}]
            By our assumption, there exists a sequence of maps $F_j \in \ff$ for which $\lim_{j \to \infty} A(F_j) = \infty$. The pull-backs $F_j^*$ map $C^\infty(\wedge^* T^* M)$ into $L^{p_0, \dots, p_k}(\wedge^* T^* \B^n_2)$, where $p_k = n/k$ for all $k \in \{0, \dots, n\}$. Thus, by Lemma \ref{lem:normalized_is_bounded}, the assumptions of Proposition \ref{prop:technical_abstraction} part \ref{enum:cohom_map_into_forms} are satisfied with $G_j = F_j^*$, $A_j = A(F_j)$, and $C = D^2$. Hence, after replacing $F_j$ with a subsequence, we find a graded homomorphism of algebras $L \colon H^*_\derham(M) \to L^{p_0, \dots, p_k}(\wedge^* T^*\B^n_2) \cap \ker(d)$ such that $A^{-1}(F_j) F_{j}^* \omega \weakto L([\omega])$ vaguely.

            Now, if $L([\omega]) \ne 0$, then the claim follows by Proposition \ref{prop:technical_abstraction} part \ref{enum:cohom_map_into_ext_alg}. Thus, it suffices to show that $A^{-1}(F_{j}) F_{j}^* \omega$ does not converge vaguely to 0. For this, let $\eta \in C^\infty_0(\B^n_2, [0, 1])$ be such that $\eta \equiv 1$ on $\B^n$. We fix $F_j$-compatible $\Sigma_j \in L^1(\B^n_2)$ for which $A(F_j, \Sigma_j) \ge A(F_j) - 1$, and note that by $F_j$-compatibility, we have $K \hodge F_{j}^* \omega + \Sigma_j \ge 0$ a.e.\ in $\B^n_2$. Using this in conjunction with the defining properties of $\ff$, we estimate that
            \begin{multline*}
                \int_{\B^n_2} \eta A^{-1}(F_j) F_{j}^* \omega\\
                \ge \frac{1}{K A(F_j)}\left( 
                    \int_{\B^n_2} \eta \cdot ( K \hodge F_{j}^* \omega 
                        + \Sigma_j) \vol_n \right) 
                    - \frac{\norm{\eta}_{L^\infty(\B^n_2)} 
                        \norm{\Sigma_j }_{L^1(\B^n_2)}}{KA(F_j)}\\
                \ge \frac{1}{K A(F_j)}\left(
                    \int_{\B^n} ( K \hodge F_{j}^* \omega 
                    	+ \Sigma_j) \vol_n \right)
                    - \frac{D}{KA(F_j)}\\
                = \frac{A(F_j, \Sigma_j) - D}{K A(F_j)} \ge \frac{1}{K} - \frac{1 + D}{K A(F_j)}.
            \end{multline*}
            Since $\lim_{j \to \infty} A(F_j) = \infty$, it follows that
            \[
                \liminf_{j \to \infty} \int_{\B^n_2} \eta A^{-1}(F_j) F_{j}^* \omega \ge K^{-1}.
            \]
            This is impossible if $A^{-1}(F_j) F_{j}^* \omega \weakto 0$ vaguely, completing the proof of the theorem.
        \end{proof}

        Thus, the proof of Theorem \ref{thm:QRcurve_with_sigma_embedding} is complete with the exception of the proof of Proposition \ref{prop:technical_abstraction}, which we provide in Section \ref{sect:abstraction_proof}.

        \section{A quasiregular value in the image implies infinite energy}
        \label{sect:proof_of_infinite_energy}

        In this section, we prove Proposition \ref{prop:QRval_infinite_energy} and Theorem \ref{thm:QRvalue_embedding}. We begin by stating a Stokes' theorem -type result for pull-backs of closed forms by a locally Sobolev map, see also e.g.\ \cite[Lemma 2.4]{Kangasniemi-Onninen_Heterogeneous}.

        \begin{lemma} \label{lem:pullback_stokes}
            Let $m\ge n\ge 2$, let $M$ be a closed, connected, oriented Riemannian $m$-manifold, and let $F\in W_\loc^{1,n}(\R^n,M)$ with $\norm{DF}_{L^n(\R^n)}<\infty$. Then, for every closed $n$-form $\omega \in C^\infty(\wedge^n T^* M)$ with $[\omega] \in K^*(M)$, we have
            \[
            \int_{\R^n} F^* \omega = 0.
            \]
        \end{lemma}

        \begin{proof}
            We first note that
            \[
            \int_{\R^n} \abs{F^* \omega} \le \norm{\omega}_{L^\infty(M)} \int_{\R^n} \abs{DF}^n < \infty,
            \]
            so the integral in the claim is well defined. We also note that, by Lemma \ref{lem:kunneth}, there exist closed forms $\alpha_1,\ldots,\alpha_j \in C^\infty(\wedge^* T^*M)$ and closed forms $\beta_1,\ldots,\beta_j \in C^\infty(\wedge^* T^*M)$ for which $\omega=\sum_{i=1}^j \alpha_i \wedge \beta_i$ and $\alpha_i \in C^\infty(\wedge^{k_i} T^*M)$, where $1\le k_i \le n-1$.

            For each $r>0$, let $\eta_r \in C_0^\infty(\R^n)$ be a test function satisfying $0\le \eta_r \le 1$, $\eta_r \equiv 1$ on $\B_r^n$, $\spt \eta_r \subset \B_{2r}^n$, and $\abs{\nabla \eta_r}\le 2/r$. Since each $F^*\alpha_i$ is weakly closed, we may then compute that
            \[
            \int_{\R^n} \eta_r F^* \omega = \sum_{i=1}^j \int_{\R^n} \eta_r F^*\alpha_i \wedge F^*\beta_i = \sum_{i=1}^j \int_{\R^n} \eta_r dT_{2r}F^*\alpha_i \wedge F^*\beta_i.
            \]
            by Proposition \ref{prop:homotopy_operator}\ref{item:homotopy}.

            By Hölder's inequality and Proposition \ref{prop:homotopy_operator}\ref{item:homotopy_continuity}, we may estimate each term in the sum by
            \begin{align*}
                \abs{ \int_{\R^n} \eta_r dT_{2r}F^*\alpha_i \wedge F^*\beta_i } &= \abs{ \int_{\R^n} d\eta_r \wedge T_{2r}F^*\alpha_i \wedge F^*\beta_i } \\
                &\le C(n) \int_{A_r} \abs{d\eta_r} \abs{T_{2r} F^*\alpha_i} \abs{F^* \beta_i} \\
                &\le C(n) \frac{1}{r} \int_{A_r} \abs{T_{2r} F^*\alpha_i} \abs{F^* \beta_i} \\
                &\le C(n) \frac{1}{r} \norm{T_{2r} F^* \alpha_i}_{L^{n/{k_i}}(A_r)} \norm{F^* \beta_i}_{L^{n/(n-k_i)}(A_r)} \\
                &\le C(n) \norm{F^* \alpha_i}_{L^{n/{k_i}}(\B_{2r}^n)} \norm{F^* \beta_i}_{L^{n/(n-k_i)}(A_r)},
            \end{align*}
            where $A_r=\B_{2r}^n \setminus \B_r^n$.

            Since $\norm{DF}_{L^n(\R^n)}<\infty$, and since $\alpha_i$ and $\beta_i$ are bounded, it follows that $F^* \alpha_i \in L^\frac{n}{k_i}(\wedge^{k_i} T^*\R^n)$ and $F^* \beta_i \in L^\frac{n}{n-k_i}(\wedge^{n-k_i} T^*\R^n)$. In particular, we obtain that each product $\norm{F^* \alpha_i}_{L^{n/{k_i}}(\B_{2r}^n)} \norm{F^* \beta_i}_{L^{n/(n-k_i)}(A_r)}$ tends to zero as $r\to \infty$. Hence,
            \[
            \lim_{r\to \infty} \int_{\R^n} \eta_r F^* \omega = 0.
            \]
            Finally, since $\eta_r F^* \omega$ tends pointwise to $F^* \omega$, and since $\norm{\omega}_{L^\infty(M)} \abs{DF}^n$ is an integrable dominant for all of the functions $\abs{\eta_r F^* \omega}$, the claim follows by dominated convergence.
        \end{proof}

        We are now ready to prove Proposition \ref{prop:QRval_infinite_energy}.

        \begin{proof}[Proof of Proposition \ref{prop:QRval_infinite_energy}]
            We may assume that $\Sigma \ge 0$ by replacing $\Sigma$ with $\abs{\Sigma}$ if necessary. Suppose towards contradiction that $y_0=f(x_0)$ for some $x_0 \in \R^n$ and $\norm{Df}_{L^n(\R^n)}<\infty$. By the single-value Reshetnyak's theorem on manifolds, see Corollary \ref{cor:reshetnyak_mflds}, we obtain that $f^{-1}\{y_0\}$ is discrete, and $i(x,f)$ is a positive integer for all $x \in f^{-1}\{y_0\}$. Notably, we may select an open bounded neighborhood $U_0$ of $x_0$ such that $\overline{U_0} \cap f^{-1}\{y_0\} = \{x_0\}$.
            
            Next, for all $r > 0$ small enough that $\B_M(y_0, r) \cap f (\partial U_0) = \emptyset$, let $\eta_r \in C_0^\infty(M)$ be a test function satisfying $0\le \eta_r \le 1$, $\eta_r \equiv 1$ on $\B_M(y_0,r/2)$, and $\spt \eta_r \subset \B_M(y_0,r)$. By the definitions of the degree and local index that we recalled in Section \ref{sect:degree_index},  we have
            \begin{multline*}
                r^n 
                \le C(M, y_0) \abs{\B_M(y_0,r/2)} 
                \le C(M, y_0) \int_M \eta_r \vol_M \\
                = \frac{C(M, y_0)}{\deg(f, y_0, U_0)} \int_{U_0} f^*(\eta_r \vol_M)
                = \frac{C(M, y_0)}{i(x_0, f)} \int_{U_0} f^*(\eta_r \vol_M).
            \end{multline*}
            On the other hand, since $M$ is not a rational cohomology sphere, Poincar\'e duality yields that the de Rham class $[\eta_r \vol_M]$ belongs to the K\"unneth ideal $K^*(M)$. Thus, by Lemma \ref{lem:pullback_stokes}, we obtain that the integral of $f^*(\eta_r \vol_M)$ over $\R^n$ vanishes. It follows that we have
            \[
            \int_{U_0} f^*(\eta_r \vol_M) = -\int_{V_r} f^*(\eta_r \vol_M) \le \int_{V_r} J_f^-,
            \]
            where $V_r=f^{-1}\B_M(y_0,r) \setminus U_0$ and $J_f^-$ denotes the negative part of the Jacobian determinant $J_f$. 
            
            To further estimate this integral, we note that the quasiregular value of $f$ at $y_0$ implies $J_f^- \le K^{-1} \dist^n(f,y_0)\Sigma$ a.e.\ in $\R^n$. Hence, we obtain
            \[
            \int_{V_r} J_f^- \le \frac{1}{K} \int_{V_r} \dist^n(f,y_0)\Sigma \le \frac{r^n}{K} \int_{V_r} \Sigma \le \frac{r^n}{K} \int_{f^{-1}\B_M(y_0,r)} \Sigma
            \]
            since $V_r \subset f^{-1}\B_M(y_0,r)$. In conclusion, we have that
            \[
                \int_{f^{-1}\B_M(y_0,r)} \Sigma \ge C(M, y_0) Ki(x_0,f),
            \]
            where $C(M, y_0) > 0$. However, the sets $f^{-1}\B_M(y_0,r)$ form a nested family of open sets whose intersection is the discrete set $f^{-1}\{y_0\}$. Therefore, since $\Sigma$ is integrable over $\R^n$, the integrals of $\Sigma$ over $f^{-1}\B_M(y_0,r)$ must tend to zero as $r\to 0$. Since $i(x_0, f) > 0$ and $K \ge 1$, it follows that $0 < C(M, y_0) K i(x_0, f) \le 0$; we have reached a contradiction, and the proof is hence complete.
        \end{proof}

        Theorem \ref{thm:QRvalue_embedding} follows as a consequence, barring the proof of Proposition \ref{prop:technical_abstraction} which remains to be proven in Section \ref{sect:abstraction_proof}.

        \begin{proof}[Proof of Theorem \ref{thm:QRvalue_embedding} assuming Proposition \ref{prop:technical_abstraction}]
            We may assume that $M$ is not a rational cohomology sphere. Since $f$ is non-constant, we have $K\ge 1$ by Corollary \ref{cor:K_less_than_1_mflds}. Hence, Proposition \ref{prop:QRval_infinite_energy} yields that $\norm{Df}_{L^n(\R^n)}=\infty$. Thus, Corollary \ref{cor:QRcurveval_embedding} implies the claim since $H^*_\derham(M)\cong H^*(M;\R)$.
        \end{proof}

        \section{Higher integrability of the weak derivative}
        \label{sect:proof_of_continuity}

        In this section, we prove Proposition \ref{prop:higher_int_general}. The proof is based on the following weak reverse Hölder type estimate; see also \cite[Lemma 4.6]{Kangasniemi-Onninen_Rescaling} and \cite[Proposition 4.1]{Heikkila_Growth_curves}.

        \begin{lemma} \label{lem:weak_reverse_hölder}
            Let $m \ge n \ge 2$, let $M$ be a closed, connected, oriented Riemannian $m$-manifold, let $\Omega \subset \R^n$ be a open, and let $F \in W^{1,n}_\loc(\Omega, M)$ satisfy \eqref{eq:QR_curve_with_Sigma}, where $K \ge 0$, $\omega \in C^\infty(\wedge^n T^* M)$ is closed and non-vanishing with $[\omega] \in K^*(M)$, and $\Sigma \in L^{1}_\loc(\Omega)$. Then, we have
            \[
            \inf_{x\in M} \abs{\omega_x}_{\mass} \dashint_Q \abs{DF}^n \le KC(\omega) \left( \dashint_{2Q} \abs{DF}^\frac{n^2}{n+1} \right)^\frac{n+1}{n} + \dashint_{2Q} 2^n \abs{\Sigma}
            \]
            for all cubes $Q$ with $2Q\subset \Omega$.
        \end{lemma}

        \begin{proof}
            Let $Q$ be a cube of side length $r$ with $2Q\subset \Omega$ and let $\eta \in C_0^\infty(\R^n)$ be a test function satisfying $0\le \eta \le 1$, $\eta \equiv 1$ on $Q$, $\spt \eta \subset 2Q$, and $\abs{\nabla \eta}\le \frac{3}{r}$. Since $F$ satisfies \eqref{eq:QR_curve_with_Sigma}, we obtain
            \[
                \inf_{x\in M} \abs{\omega_x}_{\mass} \int_Q \abs{DF}^n \le K\int_{2Q} \eta F^* \omega + \int_{2Q} \abs{\Sigma}.
            \]
            Thus it suffices to estimate the pull-back term.

            By Lemma \ref{lem:kunneth}, there exist closed forms $\alpha_1,\ldots,\alpha_j \in C^\infty(\wedge^* T^* M)$ and closed forms $\beta_1,\ldots,\beta_j \in C^\infty(\wedge^* T^* M)$ for which $\omega=\sum_{i=1}^j \alpha_i \wedge \beta_i$ and $\alpha_i \in C^\infty(\wedge^{k_i} T^* M)$, where $1\le k_i \le n-1$. It follows that
            \[
            \int_{2Q} \eta F^* \omega \le \sum_{i=1}^j \abs{\int_{2Q} \eta F^* \alpha_i \wedge F^* \beta_i}.
            \]
            Let $T$ denote the Poincar\'e homotopy operator on $2Q$, and for all $i \in \{1, \dots, j\}$, let
            \[
                p_i=\frac{n}{n-k_i}\frac{n}{n+1} 
                \quad \text{and} \quad
                q_i=\frac{n}{k_i}\frac{n}{n+1},
            \]
            noting that $1/q_i = 1/p_i^* + 1/n$. By Proposition \ref{prop:homotopy_operator}\ref{item:homotopy}, Hölder's inequality, and Proposition \ref{prop:homotopy_operator}\ref{item:homotopy_embedding} we may estimate each term in the sum by
            \begin{align*}
                \abs{\int_{2Q} \eta F^* \alpha_i \wedge F^* \beta_i} &= \abs{\int_{2Q} \eta dTF^* \alpha_i \wedge F^* \beta_i} = \abs{\int_{2Q} d\eta \wedge TF^* \alpha_i \wedge F^* \beta_i} \\
                &\le C(n) \int_{2Q} \abs{d\eta} \abs{TF^* \alpha_i} \abs{F^* \beta_i} \\
                &\le C(n) \frac{1}{r} \int_{2Q} \abs{TF^* \alpha_i} \abs{F^* \beta_i} \\
                &\le C(n) \frac{1}{r} \norm{TF^* \alpha_i}_{L^{p_i^*}(2Q)} \norm{F^* \beta_i}_{L^{p_i}(2Q)} \\
                &\le C(n) \frac{1}{r} \norm{F^* \alpha_i}_{L^{q_i}(2Q)} \norm{F^* \beta_i}_{L^{p_i}(2Q)}.
            \end{align*} 
            On the other hand, we have that
            \[
            \norm{F^* \alpha_i}_{L^{q_i}(2Q)} = \left( \int_{2Q} \abs{F^* \alpha_i}^{q_i} \right)^\frac{1}{q_i} \le \norm{\alpha_i}_{L^\infty(M)} \left( \int_{2Q} \abs{DF}^\frac{n^2}{n+1} \right)^\frac{1}{q_i}
            \]
            and similarly
            \[
            \norm{F^* \beta_i}_{L^{p_i}(2Q)} = \left( \int_{2Q} \abs{F^* \beta_i}^{p_i} \right)^\frac{1}{p_i} \le \norm{\beta_i}_{L^\infty(M)} \left( \int_{2Q} \abs{DF}^\frac{n^2}{n+1} \right)^\frac{1}{p_i}.
            \]
            
            Combining all the previous estimates, we obtain
            \[
            \inf_{x\in M} \abs{\omega_x}_{\mass} \int_Q \abs{DF}^n \le KC(\omega)\frac{1}{r} \left( \int_{2Q} \abs{DF}^\frac{n^2}{n+1} \right)^\frac{n+1}{n} + \int_{2Q} \abs{\Sigma}.
            \]
            The claim follows by taking averages.
        \end{proof}

        We obtain the following result as an immediate consequence of Lemma \ref{lem:weak_reverse_hölder} and Proposition \ref{prop:local_Gehring_lemma}.

        \begin{prop} \label{prop:gehring_integrability}
            Let $m \ge n \ge 2$, let $M$ be a closed, connected, oriented Riemannian $m$-manifold, let $\Omega \subset \R^n$ be open, and let $F \in W^{1,n}_\loc(\Omega, M)$ satisfy \eqref{eq:QR_curve_with_Sigma}, where $K \ge 0$, $\omega \in C^\infty(\wedge^n T^* M)$ is closed and non-vanishing with $[\omega] \in K^*(M)$, and $\Sigma \in L^{1+\varepsilon}_\loc(\Omega)$ for some $\varepsilon >0$. Then there exists $\lambda=\lambda(K,\omega,\varepsilon)\in (1,1+\varepsilon)$ satisfying
            \[
             \left( \dashint_Q \abs{DF}^{\lambda n} \right)^\frac{n}{\lambda(n+1)} \le C(\omega,K,\varepsilon) \left( \left( \dashint_{2Q} \abs{DF}^n \right)^\frac{n}{n+1} + \left( \dashint_{2Q} \abs{\Sigma}^\lambda \right)^\frac{n}{\lambda(n+1)} \right)
            \]
            for all cubes $Q$ for which $2Q\subset \Omega$.
        \end{prop}

        We are now ready to prove Proposition \ref{prop:higher_int_general}.

        \begin{proof}[Proof of Proposition \ref{prop:higher_int_general}]
            Since $M$ is closed, we may assume that $F$ satisfies \eqref{eq:QR_curve_with_Sigma}. By Proposition \ref{prop:gehring_integrability}, there exists $\varepsilon'>0$ for which $\abs{DF}\in L^{n+\varepsilon'}_\loc(\Omega)$. Since we also have the Sobolev embedding $W^{1,n}_\loc(\Omega, M)\subset L^{n+\varepsilon'}_\loc(\Omega, M)$, it follows that $F\in W^{1,n+\varepsilon'}_\loc(\Omega, M)$. The existence of a continuous representative follows by Morrey's inequality.
        \end{proof}

        We close this section with a remark about polynomial growth estimates for the $L^n$-norms of the derivatives of solutions of \eqref{eq:QR_with_Sigma}.

        \begin{rem}\label{rem:polynomial_growth} 
            Let $M$ be a closed, connected, oriented Riemannian $n$-manifold that is not a rational cohomology sphere, and let $f \in W^{1,n}_\loc(\R^n, M)$ satisfy \eqref{eq:QR_with_Sigma} with $K \ge 0$ and $\Sigma \in L^{1+\eps}_\loc(\R^n)$, $\eps>0$. In this case, Proposition \ref{prop:gehring_integrability} applies with $\omega = \vol_M$, and can be rearranged into an estimate of the form
            \begin{multline*}
                \norm{Df}_{L^n(2Q)}^{n^2/(n+1)} \\
                \ge r^{\frac{(1-\lambda^{-1}) n^2}{n+1}}
                \left( \frac{2^\frac{n^2}{n+1}}{C_0} 
                \norm{Df}_{L^{\lambda n}(Q)}^{n^2/(n+1)}
                - 2^{\frac{n^2(1-\lambda^{-1})}{n+1}}
                \norm{\Sigma}_{L^{\lambda}(2Q)}^{n/(n+1)}
                \right),
            \end{multline*}
            where $r$ is the side length of $Q$ and $C_0 = C_0(M, K, \eps)$. If we have 
            \begin{equation}\label{eq:poly_growth_requirement}
                \norm{Df}_{L^{\lambda n}(\R^n)}^n > 2^{-\frac{n}{\lambda}} C_0^{(n+1)/n} \norm{\Sigma}_{L^{\lambda}(\R^n)},
            \end{equation}
            then we may select a cube $Q_0 \subset \R^n$ with 
            \[
                \norm{Df}_{L^{\lambda n}(Q_0)}^n > 2^{-\frac{n}{\lambda}}
                    C_0^{(n+1)/n}\norm{\Sigma}_{L^{\lambda}(\R^n)}
            \]
            and the above estimate yields
            \[
                \int_{2Q} \abs{Df}^n
                > C_1 r^{(1-\lambda^{-1}) n}
            \] 
            for all cubes $Q$ of side length $r$ containing $Q_0$, where $C_1>0$ is such that
            \[
            C_1^\frac{n}{n+1} = \frac{2^\frac{n^2}{n+1}}{C_0} \left(  \norm{Df}_{L^{\lambda n}(Q_0)}^{n^2/(n+1)} - 2^{-\frac{n^2}{\lambda(n+1)}} C_0 \norm{\Sigma}_{L^{\lambda}(\R^n)}^{n/(n+1)} \right) > 0.
            \] 

            If $f$ is non-constant and $K$-quasiregular, then we may assume $\Sigma \equiv 0$, and hence \eqref{eq:poly_growth_requirement} is always valid. This is the polynomial growth rate estimate shown by Bonk and Heinonen in \cite[Theorem 1.11]{Bonk-Heinonen_Acta}, and it is the standard method to argue that $\norm{Df}_{L^n(\R^n)} = \infty$ in this case. However, for solutions of \eqref{eq:QR_with_Sigma}, this argument appears to require a-priori information about the relative sizes of the $L^\lambda$-norms of $\abs{Df}^n$ and $\Sigma$ before it can be used to conclude that $\norm{Df}_{L^n(\R^n)} = \infty$. The above considerations also apply to mappings $f$ with a quasiregular value at $y_0 \in M$, just with $\dist^n(f, y_0) \Sigma$ replacing $\Sigma$ in the computations.
        \end{rem}

    \section{Proof of Proposition \ref{prop:technical_abstraction}}
    \label{sect:abstraction_proof}

    In this section, we prove Proposition \ref{prop:technical_abstraction}. As stated previously, the proof is mostly a re-tread of the ideas presented in \cite{Heikkila-Pankka_Elliptic} and \cite{Heikkila_Embedding-curves}, though performed in higher generality than has been done previously. 

    For the rest of this section, let $B \subset \R^n$ be an open ball, and let $M$ be a closed, connected, oriented Riemannian $m$-manifold, where $m \ge n \ge 2$. For every $a \in \R$, we use $a_B$ and $a_M$ to denote the constant maps $x \mapsto a$ on $B$ and $M$, respectively. We fix a H\"older sequence $p_0, \dots, p_{n} \in [1, \infty]$.
    
    In order to avoid repeating ourselves, we say that a triple $(G_j, A_j, C)$ is \emph{$(p_i)$-admissible} if $G_j \colon C^\infty(\wedge^* T^* M) \to L^{p_0, \dots, p_n}(\wedge^* T^* B)$ are generalized pull-backs for all $j \in \Z_{> 0}$, $A_j \in (0, \infty)$ satisfy $\lim_{j \to \infty} A_j = \infty$, and $C \in (0, \infty)$ satisfies \eqref{eq:Aj_condition} for all $k \in \{0, \dots, n\}$, $j \in \Z_{> 0}$, and $\alpha \in C^\infty(\wedge^k T^* M)$. Given a $(p_i)$-admissible triple $(G_j, A_j, C)$, we define \emph{normalized pull-backs} $G_j^! \colon C^\infty(\wedge^* T^* M) \to L^{p_0, \dots, p_n}(\wedge^* T^* B)$
    by 
    \[
    	G_j^! \alpha = A_j^{-\frac{k}{n}} G_j(\alpha)
    \]
    for every $\alpha \in C^\infty(\wedge^k T^* M)$. Note in particular that by \eqref{eq:Aj_condition} and the fact that $G_j$ are generalized pull-backs, we have
    \begin{equation}\label{eq:norm_pullback_props}
    	\smallnorm{G_j^! \alpha}_{L^{p_k}(B)} \le C \norm{\alpha}_{L^\infty(M)} \quad \text{and} \quad
    	G_j^! d\alpha = A_j^{-\frac{1}{n}} d G_j^! \alpha.
    \end{equation}
    
    Our first step is to prove that limits of exact forms are negligible under the normalized pull-backs $G_{j}^!$, where the proof proceeds in a manner similar to \cite[Lemma 3.2]{Heikkila-Pankka_Elliptic} and \cite[Lemma 3.4]{Heikkila_Embedding-curves}.
    
    \begin{lemma} \label{lem:limit_of_exact}
    	Let $B \subset \R^n$ be an open ball, and let $M$ be a closed, connected, oriented, Riemannian $m$-manifold with $m \ge n \ge 2$. Let $p_0, \dots, p_{n} \in [1, \infty]$ be a H\"older sequence, and let $(G_j, A_j, C)$ be $(p_i)$-admissible. Then for all $k\in \{1,\ldots,n\}$ and $\alpha \in C^\infty(\wedge^{k-1} T^* M)$, we have
    	\[
    		G_j^! d\alpha \weakto 0,
    	\]
    	where the convergence is $L^{p_k}$-weak if $p_k > 1$ and vague if $p_k = 1$.
    \end{lemma}
    \begin{proof}
    	Since $G_j^! (d\alpha)$ is bounded in $L^{p_k}(\wedge^{k} T^* B)$ by \eqref{eq:norm_pullback_props}, it suffices by Lemma \ref{lem:smooth_to_weak_or_vague_conv} to check weak convergence against all test forms $\eta \in C^\infty_0(\wedge^{n-1} T^* B)$. By Hölder's inequality, \eqref{eq:norm_pullback_props}, and $\lim_{j \to \infty} A_j = \infty$, we obtain
    	\begin{multline*}
    		\abs{\int_{B} \eta \wedge G_j^! d\alpha} 
    		= A_j^{-\frac{1}{n}} \abs{\int_{B} d\eta \wedge G_j^!\alpha} 
    		\le A_j^{-\frac{1}{n}} \int_{B} \abs{d\eta} \abs{G_j^!\alpha} \\
    		\le A_j^{-\frac{1}{n}} \norm{d\eta}_{L^{p_k^*}(B)} \norm{G_j^!\alpha}_{L^{p_k}(B)} 
    		\le C A_j^{-\frac{1}{n}} \norm{d\eta}_{L^{p_k^*}(B)} \norm{\alpha}_{L^\infty(M)} \to 0
    	\end{multline*}
    	as $j\to \infty$. This concludes the proof.
    \end{proof}
    
    We fix the notation $h\colon H^*_{\derham}(M)\to C^\infty(\wedge^* T^* M)$ for the map which takes every de Rham cohomology class $c \in H^*_{\derham}(M)$ to its harmonic representative $h(c)$. Moreover, we say that, given $k \in \{0, \dots, n\}$ and a linear subspace $V \subset H^k_\derham(M)$, a linear map $L_k \colon V \to L^{p_k}(\wedge^k T^*B) \cap \ker(d)$ is a \emph{cohomology limit of $(G_j^!)$} if, for every $c \in V$ and $\omega \in c$, we have $G_j^! \omega \weakto L_k(c)$, where the convergence is $L^{p_k}$-weak convergence if $p_k \in (1, \infty]$ and vague convergence if $p_k = 1$. Now, Lemma \ref{lem:limit_of_exact} yields the following existence and uniqueness theorem for cohomology limits. 
    
    \begin{lemma}\label{lem:limit_map_construction}
    	Let $B \subset \R^n$ be an open ball, and let $M$ be a closed, connected, oriented Riemannian $m$-manifold with $m \ge n \ge 2$. Let $k \in \{0, \ldots, n\}$, let $p_0, \dots, p_{n} \in [1, \infty]$ be a H\"older sequence, and let $(G_j, A_j, C)$ be $(p_i)$-admissible.
    	\begin{enumerate}[label=(\roman*)]
    		\item \label{enum:limit_uniqueness} (Uniqueness) If $V_1, V_2 \subset H^k_\derham(M)$ are linear subspaces and $L_k^i \colon V_i \to L^{p_k}(\wedge^k T^* B) \cap \ker(d)$, $i = 1, 2$, are cohomology limits of $(G_j^!)$, then $L_k^1(c) = L_k^2(c)$ for every $c \in V_1 \cap V_2$.
    		\item \label{enum:limit_existence} (Existence) If $p_k > 1$ and $V\subset H^k_\derham(M)$ is a linear subspace, then there exists a subsequence $(G_{j_i}^!)$ of $(G_j^!)$ that has a cohomology limit $L_k \colon V \to L^{p_k}(\wedge^k T^* B) \cap \ker(d)$. 
    	\end{enumerate} 
    \end{lemma}
    \begin{proof}
    	For \ref{enum:limit_uniqueness}, let $c \in V_1 \cap V_2$ and select an arbitrary $\omega \in c$. Then we have $L_k^i(c) \in L^1(\wedge^k T^*B)$ since $B$ is finite-measured, and moreover $G_j^! \omega \weakto L_k^i(c)$ at least vaguely for $i=1,2$. Thus, $L_k^1(c) = L_k^2(c)$ a.e.\ in $B$ by the uniqueness of vague limits, which we recalled in Lemma \ref{lem:vague_limits_unique}. This completes the proof of \ref{enum:limit_uniqueness}.
    	
    	For \ref{enum:limit_existence}, let $c_1, \dots, c_{\nu}$ be a basis of $V$; note that since $M$ is closed, $H^k_\derham(M)$ is finite-dimensional, and hence $V$ has a finite basis. By \eqref{eq:norm_pullback_props}, the sequence $(G_{j}^!h(c_i))$ is bounded in $L^{p_k}(\wedge^k T^*B)$ for every $i \in \{1, \dots, \nu \}$. Since $p_k >  1$, the unit ball of $L^{p_k}(\wedge^k T^* B)$ is compact under $L^{p_k}$-weak convergence. Therefore, we find a subsequence $(G_{j_i}^!)$ such that $(G_{ j_i}^!h(c_l))$ converges weakly in $L^{p_k}(\wedge^k T^*B)$ to a form $L_k(c_l) \in L^{p_k}(\wedge^k T^*B)$ for every $l \in \{1, \dots, \nu\}$. We then define $L_k$ on $V$ by linearly extending these values $L_k(c_l)$. Since both $h$ and $G_{j_i}^!$ are linear, we obtain $L^{p_k}$-weak convergence $G_{j_i}^! h(c) \weakto L_k(c)$ for all $c \in V$.
    	
    	We then show that $L_k$ is indeed a cohomology limit of $(G_{j_i}^!)$. First, let $c \in V$. Then since $dh(c) = 0$, we have for all $\eta \in C^\infty_0(\wedge^{n-k-1} T^* B)$ that
    	\begin{multline*}
    		\int_B L_k(c) \wedge d\eta = \lim_{i \to \infty} \int_B G_{j_i}^! h(c) \wedge d\eta\\
    		= \lim_{i \to \infty} (-1)^{k+1} A_{j_i}^\frac{1}{n} \int_B G_{j_i}^! dh(c) \wedge \eta
    		= \lim_{i \to \infty} 0 = 0.
    	\end{multline*}
    	Thus, $L_k(c)$ is indeed weakly closed. Second, let $\omega \in c$. In the case $k = 0$, we in fact must have $\omega = h(c)$, and thus $G_{j_i}^! \omega \weakto L_k(c)$ trivially. If instead $k > 0$, then there exists a $\tau \in C^\infty(\wedge^{k-1} T^*M)$ for which $\omega = h(c) + d\tau$, and by Lemma \ref{lem:limit_of_exact}, we have
    	\[
    		\lim_{i \to \infty}  G_{j_i}^! \omega = \lim_{i \to \infty} (G_{j_i}^! h(c) + G_{j_i}^! d\tau) 
    		= L_k(c),
    	\]
    	where the limits are $L^{p_k}$-weak if $p_k > 1$ and vague if $p_k = 1$. The claim follows.
    \end{proof}
    
    Lemma \ref{lem:limit_map_construction} already allows us to construct the entire map $L$ of Proposition \ref{prop:technical_abstraction} in the case $p_n > 1$, in which case the only remaining step would be to show that $L$ is an algebra homomorphism. However, the fact that the existence part of Lemma \ref{lem:limit_map_construction} fails for $p_n = 1$ requires a workaround, and is the main reason why we have to consider the $n$:th layer $K^n(M)$ of the K\"unneth ideal.
    
    We then define our version of Sobolev-Poincar\'e limits, which were introduced in \cite{Heikkila-Pankka_Elliptic}. We fix a Poincar\'e homotopy operator $T$ on $B$ as in Proposition \ref{prop:homotopy_operator}. Then, given $k \in \{1, \dots, n\}$ and a linear subspace $V \subset H^k_\derham(M)$, we say that a linear map $\widehat{L}_k \colon V \to W^{d,p_k}(\wedge^{k-1} T^*B)$ is a \emph{Sobolev-Poincar\'e limit of $(G_j^!)$} if $d\widehat{L}_k$ is a cohomology limit of $(G_j^!)$ and $T G_j^! h(c) \to \widehat{L}_k(c)$ strongly in the $L^{p_k}$-sense for every $c \in V$.
    
    Next, similarly to \cite[Lemma 4.1]{Heikkila-Pankka_Elliptic} and \cite[Lemma 4.1]{Heikkila_Embedding-curves}, we show that if $(G_j^!)$ has a cohomology limit, then a subsequence of $(G_j^!)$ has a Sobolev-Poincar\'e limit.
    
    \begin{lemma} \label{lem:weak_is_exact}
    	Let $B \subset \R^n$ be an open ball, and let $M$ be a closed, connected, oriented Riemannian $m$-manifold with $m \ge n \ge 2$. Let $p_0, \dots, p_{n} \in [1, \infty]$ be a H\"older sequence, and let $(G_j, A_j, C)$ be $(p_i)$-admissible. If $(G_j^!)$ has a cohomology limit $L_k \colon V \to L^{p_k}(\wedge^k T^* B) \cap \ker(d)$, where $k \in \{1, \dots, n\}$ and $V \subset H^k_{\derham}(M)$, then there exists a subsequence $(G_{j_i}^!)$ of $(G_{j}^!)$ with a Sobolev-Poincar\'e limit $\widehat{L}_k \colon V \to W^{d, p_k}(\wedge^{k-1} T^* B)$ satisfying $d\widehat{L}_k = L_k$.
    \end{lemma}
    
    \begin{proof}
    	Let $c_1,\ldots,c_\nu$ be a basis of $V$. Then for every $i \in \{1, \dots, \nu\}$, \eqref{eq:norm_pullback_props} yields that $\smallnorm{G_j^! h(c_i)}_{L^{p_k}(B)} \le C \norm{h(c_i)}_{L^\infty(M)}$ for all $j \in \Z_{> 0}$. Since $T$ is compact from $L^{p_{k}}(\wedge^{k} T^* B)$ to $L^{p_{k}}(\wedge^{k-1} T^* B)$ by Proposition \ref{prop:homotopy_operator}\ref{item:homotopy_compactness}, there  exist a subsequence $(G_{j_i}^!)$ and forms $\tau_1,\ldots,\tau_\nu \in L^{p_k}(\wedge^{k-1} T^* B)$ for which $T G_{j_i}^! h(c_l) \to \tau_l$ strongly in the $L^{p_k}$-norm as $j \to \infty$. 
    	
    	Thus, we define $\widehat{L}_k(c_l) =\tau_l$ and extend linearly, obtaining by linearity of $h$, $G_{j_i}^!$, and $T$ that $T G_{j_i}^! h(c) \to \widehat{L}_k(c)$ strongly in the $L^{p_k}$-norm for all $c \in V$. Moreover, if $c \in V$, Proposition \ref{prop:homotopy_operator}\ref{item:homotopy} and $d G_j^!h(c) = 0$ yield that $d T G_j^! h(c) = G_j^! h(c)$, and consequently 
    	\begin{multline*}
    		\int_{B} L_k(c) \wedge \eta 
    		= \lim_{i\to \infty} \int_{B} G_{j_i}^! h(c) \wedge \eta \\
    		= (-1)^{k+1} \lim_{i\to \infty} \int_{B} T G_{j_i}^! h(c) \wedge d\eta 
    		= (-1)^{k+1} \int_{B} \widehat{L}_k(c) \wedge d\eta
    	\end{multline*}
    	for all $\eta \in C_0^\infty(\wedge^{n-k-1} T^* B)$. Hence, $d\widehat{L}_k(c)=L_k(c)$ weakly. This implies both that $d\widehat{L}_k = L_k$, and that $d\widehat{L}_k$ is a cohomology limit of $(G_{j_i}^!)$, completing the proof of the claim.
    \end{proof}
    
    We then, as in \cite[Lemma 4.2]{Prywes_Annals} and \cite[Proposition 4.2]{Heikkila-Pankka_Elliptic}, show that the weak exterior derivative of a Sobolev–Poincar\'e limit commutes with the exterior product in a weak sense.
    
    \begin{lemma}\label{lem:weak_commutation_with_wedge}
    	Let $B \subset \R^n$ be an open ball, and let $M$ be a closed, connected, oriented Riemannian $m$-manifold with $m \ge n \ge 2$. Let $p_0, \dots, p_{n} \in [1, \infty]$ be a H\"older sequence, and let $(G_j, A_j, C)$ be $(p_i)$-admissible. Suppose that $k_1, k_2 \in \{1, \ldots, n\}$ with $k_1 + k_2 := k \le n$, $V_i \subset H^{k_i}_\derham(M)$ are linear subspaces, and $\widehat{L}_{k_i} \colon V_i \to W^{d, p_{k_i}}(\wedge^{k_i - 1} T^* B)$ are Sobolev-Poincar\'e limits of $(G_j^!)$, where $i=1,2$. Then for all $c_1 \in V_1$, $c_2 \in V_2$, and $\omega \in c_1 \wedge c_2$, we have
    	\[
    		G_j^! \omega \weakto d\widehat{L}_{k_1}(c_1) \wedge d\widehat{L}_{k_2}(c_2).
    	\]
    	Here, the convergence is $L^{p_{k}}$-weak convergence if $p_k > 1$, and vague convergence if $p_k = 1$.
    \end{lemma}
    
    \begin{proof}
    	Let $c_1 \in V_1$, $c_2 \in V_2$, and $\omega \in c_1 \wedge c_2$. By \eqref{eq:norm_pullback_props}, and Lemma \ref{lem:smooth_to_weak_or_vague_conv}, it suffices to prove that
    	\[
    		\lim_{j\to \infty} \int_{B} \eta \wedge 
    			\left( G_j^! \omega - d\widehat{L}_{k_1}(c_1) \wedge d\widehat{L}_{k_2}(c_2) \right) 
    		= 0
    	\]
    	for all $\eta \in C_0^\infty(\wedge^{n-k} T^* B)$. Moreover, since $\omega \in c_1 \wedge c_2$, we have $\omega = h(c_1) \wedge h(c_2) + d\tau$ for some $\tau \in C^\infty(\wedge^{k-1} T^* B)$. Thus, due to Lemma \ref{lem:limit_of_exact}, and since $G_j^!$ respect the wedge product, it suffices to show that
    	\[
    		\lim_{j\to \infty} \int_{B} \eta \wedge 
    			\left( G_j^! h(c_1) \wedge G_j^! h(c_2) 
    				- d\widehat{L}_{k_1}(c_1) \wedge d\widehat{L}_{k_2}(c_2)\right)  
    		= 0
    	\]
    	for all $\eta \in C_0^\infty(\wedge^{n-k} T^* B)$.
    	
    	By the triangle inequality, it suffices to find an upper bound tending to zero for
    	\begin{multline}\label{eq:terms_to_estimate_wedge}
    	\abs{ \int_{B} \eta \wedge (d\widehat{L}_{k_1} (c_1) - G_j^! h(c_1)) \wedge d\widehat{L}_{k_2}(c_2) }\\
    	+ \abs{ \int_{B} \eta \wedge G_j^! h(c_1) \wedge (d\widehat{L}_{k_2}(c_2) - G_j^! h(c_2)) }.
    	\end{multline}
    	To show this for the first term in \eqref{eq:terms_to_estimate_wedge}, we use Proposition \ref{prop:homotopy_operator}\ref{item:homotopy} to obtain that
    	\begin{multline*}
    		\abs{ \int_{B} \eta \wedge (d\widehat{L}_{k_1} (c_1) - G_j^! h(c_1)) 
    			\wedge d\widehat{L}_{k_2}(c_2)  } \\
    		= \abs{ \int_{B} \eta \wedge d (\widehat{L}_{k_1} (c_1) - T G_j^! h(c_1)) 
    			\wedge d\widehat{L}_{k_2}(c_2) } \\
    		= \abs{ \int_{B} d\eta \wedge(\widehat{L}_{k_1} (c_1) - T G_j^! h(c_1)) 
    			\wedge d\widehat{L}_{k_2}(c_2) } \\
    		\le C(n) \norm{d\eta}_{L^\infty(\B_2^n)} \int_{B} \abs{ d\widehat{L}_{k_2}(c_2)}
    			\abs{\widehat{L}_{k_1} (c_1) - T G_j^! h(c_1)}.
    	\end{multline*}
    	Then, by \eqref{eq:Holder_seq_est} along with the inequality $\norm{\varphi}_{L^1(B)} \le C(B) \norm{\varphi}_{L^{p_k}(B)}$ for all $\varphi \in L^1(B)$, we further have
    	\begin{multline*}
    		\int_{B} \abs{ d\widehat{L}_{k_2}(c_2)}
    			\abs{\widehat{L}_{k_1} (c_1) - T G_j^! h(c_1)} \\
    		\le C(B) \norm{d\widehat{L}_{k_2}(c_2)}_{L^{p_{k_2}}(B)} 
    			\norm{\widehat{L}_{k_1} (c_1) - T G_j^! h(c_1)}_{L^{p_{k_1}}(B)}.
    	\end{multline*}
    	Since $TG_j^! h(c_1) \to \widehat{L}(c_1)$ strongly in $L^{p_{k_1}}(\wedge^{k_1} T^* B)$, this upper bound tends to zero as $j \to \infty$.
    	
    	For the second term in \eqref{eq:terms_to_estimate_wedge}, a completely analogous argument to the above yields
    	\begin{multline*}
    		\abs{ \int_{B} \eta \wedge G_j^! h(c_1) \wedge (d\widehat{L}_{k_2}(c_2) - G_j^! h(c_2)) }\\
    		\le C(B) \norm{d\eta}_{L^\infty(\B_2^n)} \norm{G_j^! h(c_1)}_{L^{p_{k_1}}(B)}
    			\norm{\widehat{L}_{k_2} (c_2) - T G_j^! h(c_2)}_{L^{p_{k_2}}(B)}.
    	\end{multline*}
    	Thus, with the additional use of \eqref{eq:norm_pullback_props} which shows that $\smallnorm{G_j^! h(c_1)}_{L^{p_{k_1}}(B)}$ is uniformly bounded in $j$, this upper bound also tends to zero as $j \to \infty$. Hence, the claim follows.
    \end{proof}
    
    The first consequence of Lemma \ref{lem:weak_commutation_with_wedge} is that cohomology limits commute with the wedge product; see also the corresponding results in \cite[Lemma 4.3]{Heikkila-Pankka_Elliptic} and \cite[Lemma 4.3]{Heikkila_Embedding-curves}.
    
    \begin{lemma}\label{lem:commutation_with_wedge}
    	Let $B \subset \R^n$ be an open ball, and let $M$ be a closed, connected, oriented Riemannian $m$-manifold with $m \ge n \ge 2$. Let $p_0, \dots, p_{n} \in [1, \infty]$ be a H\"older sequence, and let $(G_j, A_j, C)$ be $(p_i)$-admissible. Let $k_0, k_1, k_2 \in \{0, \dots, n\}$ with $k_1 + k_2 = k_0$, and let $L_{k_i} \colon V_i \to L^{p_{k_i}}(\wedge^{k_i} T^* B) \cap \ker(d)$ be cohomology limits of $(G_j^!)$, where $V_i \subset H^{k_i}_\derham(M)$, $i = 0, 1, 2$. Then 
    	\[
    		L_{k_0}(c_1 \wedge c_2) = L_{k_1}(c_1) \wedge L_{k_2}(c_1)
    	\]
    	for all $c_1 \in V_1$ and $c_2 \in V_2$ with $c_1 \wedge c_2 \in V_0$.
    \end{lemma}
    
    \begin{proof}
    	Consider first the case $k_1 = 0$. In this case, $h(c_1) \in C^\infty(\wedge^0 T^* M) = C^\infty(M, \R)$ is a constant function $a_M$ for some $a \in \R$, and thus $c_1 \wedge c_2 = ac_2$. We note that by $\R$-linearity of $G_j$ and the assumption $G_j(1_M) = 1_B$, we have $G_j^!(h(c_1)) = G_j(a_M) = a G_j(1_M) = a_B$ for every $j$. Since $G_j^!(h(c_1)) \weakto L_{k_1}(c_1)$, we have that $L_{k_1}(c_1) = a_B$; see again Lemma \ref{lem:vague_limits_unique}. Thus, using the $\R$-linearity of cohomology limits and the uniqueness-part of Lemma \ref{lem:limit_map_construction}, we have
    	\[
    		L_{k_0}(c_1 \wedge c_2) = L_{k_0}(a c_2) = a L_{k_0}(c_2) = a L_{k_2}(c_2) = L_{k_1}(c_1) \wedge L_{k_2}(c_2).
    	\] 
    	This completes the proof when $k_1 = 0$; the case $k_2 = 0$ is analogous.
    	
    	Thus, it remains to prove the claim when $k_1, k_2 \ge 1$. In this case, by using Lemma \ref{lem:weak_is_exact}, we find a subsequence $G_{j_i}$ of $G_j$ which has Sobolev-Poincar\'e limits $\widehat{L}_{k_i}$ with $d\widehat{L}_{k_i} = L_{k_i}$ for $i = 1, 2$. By using Lemma \ref{lem:weak_commutation_with_wedge} and the fact that $L_{k_0}$ is a cohomology limit of $(G_j^!)$, we thus have
    	\[
    		L_{k_0}(c_1 \wedge c_2) = \lim_{i \to \infty} G_{j_i}^! h(c_1 \wedge c_2)
    		= d\widehat{L}_{k_1}(c_1) \wedge d\widehat{L}_{k_2}(c_2) = L_{k_1}(c_1) \wedge L_{k_2}(c_2);
    	\]
    	here, all limits are at least vague, and hence unique by Lemma \ref{lem:vague_limits_unique}. The claim follows.
    \end{proof}
    
    The second main consequence of Lemma \ref{lem:weak_commutation_with_wedge} is that it allows us to construct a cohomology limit $L_n \colon K^n(M) \to L^1(\wedge^n T^* B)$ in the case $p_n = 1$.
    
    \begin{lemma} \label{lem:extension_to_K}
    	Let $B \subset \R^n$ be an open ball, and let $M$ be a closed, connected, oriented Riemannian $m$-manifold with $m \ge n \ge 2$. Let $p_0, \dots, p_{n} \in [1, \infty]$ be a H\"older sequence, and let $(G_j, A_j, C)$ be $(p_i)$-admissible. Let $L_k \colon H^k_\derham(M) \to L^{p_k}(\wedge^k T^* B)$, $k \in \{1, \dots, n-1\}$, be cohomology limits of $(G_j^!)$. Then there exists a subsequence $(G_{j_i}^!)$ of $(G_j^!)$ with a cohomology limit $L_n \colon K^n(M) \to L^{p_n}(\wedge^k T^* B) \cap \ker(d)$.
    \end{lemma}

    \begin{proof}
    	Note that $H^n_\derham(M)$ and hence also $K^n(M)$ are finite-dimensional. Moreover, $K^n(M)$ is the linear span of elements of the form $e \wedge f$, where $e \in H^{k}_\derham(M)$ and $f \in H^{n-k}_\derham(M)$ for some $k \in \{1, \dots, n-1\}$. Since every generating set of a linear space contains a basis, we may fix a basis $c_1, \dots, c_\nu$ of $K^n(M)$ consisting of elements of the form  $c_i = e_i \wedge f_i$, where $e_i \in H^{k_i}_\derham(M)$ and $f_i \in H^{n-k_i}_\derham(M)$ for some $k_i \in \{1, \dots, n-1\}$. We define
    	\[
    		L_n(c_i) = L_{k_i}(e_i) \wedge L_{n-k_i}(f_i)
    	\]
    	for $i \in \{1, \dots, \nu \}$ and extend linearly, obtaining a linear map $L_n \colon K^n(M) \to L^{p_n}(\wedge^n T^* B)$. We note that since the exterior product of weakly closed forms in $L^{p_0, \dots, p_k}(\wedge^* T^*B)$ is weakly closed, it follows that $L_n(c) \in L^{p_n}(\wedge^k T^* B) \cap \ker(d)$ for every $c \in K^n(M)$. 
    	
    	We then use Lemma \ref{lem:weak_is_exact} repeatedly to find a subsequence $(G_{j_i}^!)$ of $(G_j^!)$ such that for all $k \in \{1, \dots, n-1\}$, $(G_{j_i}^!)$ has a Sobolev-Poincar\'e limit $\widehat{L}_k \colon H^k_\derham(M) \to W^{d,p_k}(\wedge^{k-1} T^* B)$ with $d\widehat{L}_k = L_k$. Now, suppose that $c \in K^n(M)$ and $\omega \in c$, in which case we can write
    	\[
    		\omega = a_1 h(c_1) + \dots + a_\nu h(c_\nu) + d\tau
    	\]
    	where $a_l \in \R$ for all $l \in \{1, \dots \nu\}$ and $\tau \in C^{\infty}(\wedge^{n-1} M)$. Since $G_{j_i}^!d\tau \weakto 0$ by Lemma \ref{lem:limit_of_exact}, and since $G_{j_i}^!h(c_l) \weakto L_n(c_l)$ for all $l \in \{1, \dots, \nu\}$ by Lemma \ref{lem:weak_commutation_with_wedge} and the identity $d\widehat{L}_k = L_k$, it follows that $G_{j_i}^!(\omega) \weakto L_n(c)$, where the convergence is $L^{p_n}$-weak if $p_n > 1$ and vague if $p_n = 1$. That is, $L_n$ is a cohomology limit of $(G_{j_i}^!)$, completing the proof.
    \end{proof}
    
    We are now ready to prove Proposition \ref{prop:technical_abstraction}.
    
    \begin{proof}[Proof of Proposition \ref{prop:technical_abstraction}]
    	For part \ref{enum:cohom_map_into_forms}, suppose first that $p_n > 1$. In this case, we use Lemma \ref{lem:limit_map_construction} to find a subsequence $(G_{j_i})$ of $(G_j)$ such that $(G_{j_i}^!)$ has a cohomology limit $L_k \colon H^k_\derham(M) \to L^p(\wedge^k T^* M) \cap \ker(d)$ for all $k \in \{1, \dots, n\}$. Now, for all $c \in H^k_\derham(M)$, $k \in \{0, \dots, m\}$, we define
    	\[
    		L(c) = \begin{cases}
    			L_k(c) & k \le n,\\
    			0 & k > n,
    		\end{cases}
    	\]
    	obtaining a graded linear map $H^*_\derham(M) \to L^{p_0, \dots, p_k}(\wedge^* T^* B) \cap \ker(d)$. Since $L_k$ are cohomology limits of $(G_{j_i}^!)$, it remains to check that $L$ commutes with the wedge product. For this, suppose that $c_i \in H^{k_i}_\derham(M)$ with $i = 1, 2$. Then if $k_1 + k_2 \le n$, we have $L(c_1 \wedge c_2) = L(c_1) \wedge L(c_2)$ by Lemma \ref{lem:commutation_with_wedge}, and if $k_1 + k_2 > n$, we have $L(c_1 \wedge c_2) = 0 = L(c_1) \wedge L(c_2)$ trivially. Thus, $L$ satisfies all desired properties.
    	
    	We then consider the slightly trickier case $p_n = 1$. In this case, we use Lemma \ref{lem:limit_map_construction} followed by Lemma \ref{lem:extension_to_K} to find a subsequence $(G_{j_i})$ of $(G_j)$ such that $(G_{j_i}^!)$ has cohomology limits $L_k \colon H^k_\derham(M) \to L^p(\wedge^k T^* M) \cap \ker(d)$ for $k \in \{1, \dots, n-1\}$ and $L_n \colon K^n(M) \to L^p(\wedge^k T^* M) \cap \ker(d)$. We fix any linear projection $\pi \colon H^n_\derham(M) \to K^n(M)$, and define for all $c \in H^k_\derham(M)$, $k \in \{0, \dots, m\}$, that
    	\[
    		L(c) = \begin{cases}
    			L_k(c) & k < n,\\
    			L_n(\pi(c)) & k = n,\\
    			0 & k > n,
    		\end{cases}
    	\]
    	again obtaining a graded linear map $H^*_\derham(M) \to L^{p_0, \dots, p_k}(\wedge^* T^* B)  \cap \ker(d)$. 
    	
    	Since $L_k$ are cohomology limits of $(G_{j_i}^!)$, and since $\pi(c) = c$ for $c \in K^n(M)$, it again only remains to check that $L$ respects the wedge product. Thus, suppose that $c_i \in H^{k_i}_\derham(M)$ with $i = 1, 2$. If $k_1 + k_2 < n$, then Lemma \ref{lem:commutation_with_wedge} again yields $L(c_1 \wedge c_2) = L(c_1) \wedge L(c_2)$, and if $k_1 + k_2 > n$, we again trivially have $L(c_1 \wedge c_2) = 0 = L(c_1) \wedge L(c_2)$. In the remaining case $k_1 + k_2 = n$, if both $k_1 > 0$ and $k_2 > 0$, then $c_1 \wedge c_2 \in K^n(M)$. Hence, Lemma \ref{lem:commutation_with_wedge} yields
    	\[
    		L(c_1) \wedge L(c_2) = L_n(c_1 \wedge c_2) = L_n(\pi(c_1 \wedge c_2)) = L(c_1 \wedge c_2).
    	\]
    	By symmetry, the final case to consider is $k_1 = 0, k_2 = n$. In this case, $h(c_1)$ is a constant function $a_M$, $a \in \R$. Recall that in this case, $a_B = G_{j_i}^!(a_M) \weakto L_0(c_1)$ vaguely, which by the uniqueness of vague limits from Lemma \ref{lem:vague_limits_unique} implies that $L_0(c_1) = a_B$. Thus,
    	\[
    		L(c_1 \wedge c_2) = a L(c_2) = L_0(c_1) \wedge L(c_2),
    	\]
    	completing the proof of part \ref{enum:cohom_map_into_forms}.
    	
    	For part \ref{enum:cohom_map_into_ext_alg}, we suppose that $\omega \in C^\infty(\wedge^k T^* M)$ is closed and $L([\omega])$ is not a.e.\ vanishing, with our objective being to construct a graded homomorphism of algebras $\Phi \colon H^*_\derham(M) \to \wedge^* \R^n$ with $\Phi([\omega]) \ne 0$. Note that we must have that $[\omega] \ne 0$. Thus, we may select a graded linear basis $c_1,\ldots,c_\nu$ of $H_\derham^*(M)$ so that $[\omega]$ is one of the basis elements. 
    	
    	We fix Borel representatives for $Lc_i$ and $L(c_i \wedge c_l)$, where $i,l=1,\ldots,\nu$. Let
    	\[
    		E = \left\{ x\in B \colon (L([\omega]))_x  \ne 0 \right\}
    	\]
    	and
    	\[
    	E_{i,l} = \{ x\in B^n \colon (Lc_i)_x \wedge (Lc_l)_x = (L(c_i \wedge c_l))_x \}
    	\]
    	for $i,l=1,\ldots,\nu$. Now $E_{i,l}$ are sets of full measure for all $(i,l)\in \{1,\ldots,\nu\}^2$, and $E$ has positive measure. Thus, there exists a $x_0 \in E \cap \bigcap_{i,l=1}^\nu E_{i,l}$.
    	We then let $\Phi \colon H_\derham^*(M)\to \wedge^* \R^n$ be the linear map defined by 
    	\[
    		\Phi(c_i) = (Lc_i)_{x_0} \qquad \text{for } i=1,\ldots,\nu.
    	\]
    	By definition, the map $\Phi$ is graded. Moreover, we have
    	\[
    	\Phi(c_i \wedge c_l) = (L(c_i \wedge c_l))_{x_0} = (Lc_i)_{x_0} \wedge (Lc_l)_{x_0} = \Phi (c_i) \wedge \Phi (c_l)
    	\]
    	for all $(i,l)\in \{1,\ldots,\nu\}^2$, since $x_0 \in E_{i,l}$. Hence, the map $\Phi$ is an algebra homomorphism by linearity. Moreover, since $[\omega]$ is a basis element, we have
    	\[
    		\Phi ([\omega]) = (L([\omega]))_{x_0} \ne 0
    	\]
    	since $x_0 \in E$.

        Finally, we point out why $\Phi$ is injective when $k = m = n$. In this case, since $\Phi([\omega]) \ne 0$, we must have $[\omega] = \lambda [\vol_M]$ for some $\lambda \ne 0$. Now, if we suppose towards contradiction that we have $\Phi(c) = 0$ with $c \in H^l_\derham(M) \setminus \{0\}$ for some $l \in \{0, \dots, m\}$, then Poincar\'e duality yields a $c' \in H^{m-l}_\derham(M)$ with $c \wedge c' = [\omega]$. As $\Phi$ is a homomorphism of algebras, a contradiction follows since $0 \ne \Phi([\omega]) = \Phi(c) \wedge \Phi(c') = 0$. Thus, $\Phi$ is injective if $k = m = n$.
    \end{proof}

    \appendix

    \section{Compactness of the Poincar\'e homotopy operator}\label{sect:Poincare_homotopy_compactness}
    
    In this appendix, we provide a proof of Proposition \ref{prop:homotopy_operator} \ref{item:homotopy_compactness}. We state it here in a slightly more general form.
    
    \begin{prop}\label{prop:compact_embedding}
    	Let $D \subset \R^n$ be a bounded, convex domain. Then for all $k \in \{1, \dots, n\}$ and $p, q \in [1, \infty]$ with $q^{-1} < p^{-1} + n^{-1}$, the Poincar\'e homotopy operator $T \colon L^q(\wedge^k T^* D) \to L^p(\wedge^{k-1} T^* D)$ is compact.
    \end{prop}
    
    The case $p = q$ of this result, which is the only case we require, is stated in \cite[Remark 4.1]{Iwaniec-Lutoborski}; however, we do not see how the result follows from the explanation given there. In the cases $1 < q < \infty$, as noted in \cite[p.\ 235]{Bonk-Heinonen_Acta}, one can easily derive this result using \cite[Proposition 4.1]{Iwaniec-Lutoborski} and the compactness of the embedding $W^{1,q}(D) \hookrightarrow L^p(D)$. However, at the level of generality in which we state Proposition \ref{prop:technical_abstraction}, all cases $1 \le q \le \infty$ are utilized in the proof. For this reason, we have elected to provide an exposition of the proof of Proposition \ref{prop:compact_embedding} in this appendix.
    
    \subsection{Preliminaries}
    
    We start by recalling a few preliminary results. The first one is the Kolmogorov-Riesz-Fr\'echet theorem, which describes compact sets in $L^p$-spaces; see e.g.\ \cite[Theorem 4.26 and Corollary 4.27]{Brezis_FunctionalAnalysisAndPDEs}. We state it here for $k$-forms, for which the result follows immediately by coordinate-wise application of the original result for real-valued functions. Here and in what follows, for every $h \in \R^n$, we let $\tau_h \colon L^p(\wedge^k T^*\R^n) \to L^p(\wedge^k T^*\R^n)$ denote the translation operation defined by $(\tau_h \omega)_x = \omega_{x + h}$.
    
    \begin{thm}[Kolmogorov-Riesz-Fr\'echet]\label{thm:Kolmogorov-Frechet-Riesz}
    	Let $1 \le p < \infty$ and let $K$ be a bounded subset of $L^p(\wedge^k T^* \R^n)$. Then $\overline{K}$ is compact if the two following properties hold:
    	\begin{enumerate}[label=(\roman*)]
    		\item \emph{$L^p$-equicontinuity}: for every $\eps > 0$, there exists a $\delta > 0$ such that for all $\omega \in K$, $\norm{\omega - \tau_h \omega}_{L^p(\R^n)} < \eps$ when $h \in \B^n(0, \delta)$.
    		\item \emph{$L^p$-equitightness}: for every $\eps > 0$, there exists a $r > 0$ such that for all $\omega \in K$, $\norm{\omega}_{L^p(\R^n\setminus \mathbb{B}(0, r))} < \eps$.
    	\end{enumerate}
    \end{thm}
    
    Next, we recall a convergence result for $L^p$-bounded sequences that tend pointwise to zero.
    
    \begin{lemma}\label{lem:pseudo_dominated_conv}
    	Let $\Omega \subset \R^n$ be a bounded open set, let $1 \le q < p \le \infty$, and let $f_j \colon \Omega \to \R$ be a sequence of measurable functions satisfying $\lim_{j \to \infty} f_j(x) = 0$ for a.e.\ $x \in \Omega$, and
    	\[
    		\limsup_{j \to \infty} \norm{f_j}_{L^p(\Omega)} = C < \infty.
    	\]
    	Then
    	\[
    		\lim_{j \to \infty} \norm{f_j}_{L^q(\Omega)} = 0.
    	\]
    \end{lemma}
    \begin{proof}
    	Let $\eps > 0$. By Egorov's theorem, there exists an $A \subset \Omega$ with measure $\abs{A} < \eps$ and $\lim_{j \to \infty} f_j(x) = 0$ uniformly on $\Omega \setminus A$. Thus, by H\"older's inequality,
    	\begin{multline*}
    		\limsup_{j \to 0} \norm{f_j}_{L^q(\Omega)} 
    		\le \limsup_{j \to 0} \norm{f_j}_{L^q(\Omega \setminus A)} + \limsup_{j \to 0} \norm{f_j}_{L^q(A)}\\
    		= \limsup_{j \to 0} \norm{f_j}_{L^q(A)}
    		\le \abs{A}^{\frac{1}{q} - \frac{1}{p}} \limsup_{j \to 0} \norm{f_j}_{L^p(X)}
    		\le C \eps^{\frac{1}{q} - \frac{1}{p}} .
    	\end{multline*}
    	Now, using our assumption $p > q$, the claim follows by letting $\eps \to 0$.
    \end{proof}
    
    Following this, we recall an integral estimate obtained by applying the usual proof of Young's convolution inequality on a non-convolution kernel.
    
    \begin{lemma}\label{lem:pseudo_young}
    	Let $\Omega \subset \R^n$ be measurable, let $p, q, \gamma \in [1, \infty]$ be such that $q^{-1} + \gamma^{-1} = p^{-1} + 1$, and let $\phi \colon \Omega \to [-\infty, \infty]$ and $\psi \colon \Omega \times \R^n \to [-\infty, \infty]$ be measurable. We denote $\psi^x(z) = \psi_z(x) = \psi(z, x)$, and define a function $\Phi \colon \R^n \to (-\infty, \infty]$ by setting
    	\[
    	\Phi(x) = \int_\Omega \phi(z) \psi(z, x) \, \dd z
    	\]
    	whenever $x \in \R^n$ is such that $\phi \cdot \psi^x \in L^1(\Omega)$, with $\Phi(x) = \infty$ otherwise. Then
    	\[
    	\norm{\Phi}_{L^p(\R^n)} 
    	\le \left(\esssup_{x \in \R^n} \norm{\psi^x}_{L^\gamma(\Omega)}\right)^{1 - \frac{\gamma}{p}}
    	\left(\esssup_{z \in \Omega} \norm{\psi_z}_{L^\gamma(\R^n)}\right)^\frac{\gamma}{p}
    	\norm{\phi}_{L^q(\Omega)},
    	\]
    	and if $p = \infty$, then we also have
    	\[
    		\sup_{x \in \R^n} \abs{\Phi(x)} \le \left(\sup_{x\in \R^n} \norm{\psi^x}_{L^\gamma(\Omega)}\right) \norm{\phi}_{L^q(\Omega)}.
    	\]
    \end{lemma}
    \begin{proof}
    	We first cover the case $p = \infty$, which follows with just a simple use of H\"older's inequality, yielding
    	\[
    		\abs{\Phi(x)} \le \int_\Omega \abs{\phi} \abs{\psi^x} \, dz
    		\le \norm{\psi^x}_{L^\gamma(\Omega)} \norm{\phi}_{L^q(\Omega)},
    	\]
    	for all $x \in \R^n$. Next, suppose $p < \infty$, in which case also $q < \infty$ and $\gamma < \infty$ due to $q^{-1} + \gamma^{-1} = 1 + p^{-1} > 1$. We denote the corresponding H\"older conjugates of $p, q, \gamma$ by $p^*, q^*, \gamma^*$, and note that $q/p + q/\gamma^* = 1$, $\gamma/p + \gamma/q^* = 1$, and $1/q^* + 1/\gamma^* + 1/p = 1$. 
    	
    	We then have via a 3-term H\"older's inequality that
    	\begin{multline*}
    		\abs{\Phi(x)}^p \le \left(\int_\Omega \abs{\phi} \abs{\psi^x} \right)^p
    		= \left(\int_\Omega \abs{\phi}^{\frac{q}{p}} \abs{\psi^x}^{\frac{\gamma}{p}} 
    			\cdot \abs{\phi}^{\frac{q}{\gamma^*}} \cdot \abs{\psi^x}^{\frac{\gamma}{q^*}}\right)^p\\
    		\le \norm{\psi^x}_{L^{\gamma}(\Omega)}^{p\gamma/q^*} \norm{\phi}_{L^q(\Omega)}^{pq/\gamma^*} \int_\Omega \abs{\phi}^q \abs{\psi^x}^\gamma.
    	\end{multline*}
    	Thus,
    	\begin{multline*}
    	\norm{\Phi}_{L^p(\R^n)}\\
    	\le \norm{\phi}_{L^q(\Omega)}^{q/\gamma^*} \left(\esssup_{x \in \R^n} \norm{\psi^x}_{L^\gamma(\Omega)}\right)^{\frac{\gamma}{q^*}}
    		\left( \int_{\R^n} \int_\Omega \abs{\phi(z)}^q\abs{\psi^x(z)}^\gamma \, dz \, dx\right)^{\frac{1}{p}}.
    	\end{multline*}
    	Fubini's theorem then yields
    	\begin{multline*}
    		\left( \int_{\R^n} \int_\Omega \abs{\phi(z)}^q\abs{\psi^x(z)}^\gamma \, dz \, dx\right)^{\frac{1}{p}}
    		= \left( \int_\Omega \abs{\phi(z)}^q \norm{\psi_z}_{L^\gamma(\R^n)}^{\gamma} \, dz \right)^{\frac{1}{p}}\\
    		\le \left(\esssup_{z \in \Omega} \norm{\psi_z}_{L^\gamma(\R^n)}\right)^\frac{\gamma}{p} \norm{\phi}_{L^q(\Omega)}^{q/p},
    	\end{multline*}
    	and the claim follows.
    \end{proof}
    
    We also recall and fix notation for the interior product of $k$-covectors on $\R^n$. Namely, if $\alpha \in \wedge^k (\R^n)^*$ is a $k$-covector with $k \ge 1$, and $v \in \R^n$ is a vector, then we define the \emph{interior product} $\alpha \iprodl v \in \wedge^{k-1} (\R^n)^*$ by
    \[
    (\alpha \iprodl v)(w_1, \dots, w_{k-1}) = \alpha(v, w_1, \dots, w_{k-1})
    \]
    for all $w_1, \dots, w_{k-1} \in \R^n$.
    
    \subsection{Definition of the Poincar\'e homotopy operator}
    
    Although we call $T \colon L^q(\wedge^k T^* D) \to L^p(\wedge^{k-1} T^* D)$ the Poincar\'e homotopy operator, this is in fact a slight abuse of terminology, as the construction involves a choice which affects the values of the operator. Namely, in the definition, one fixes a non-negative smoothing kernel $\varphi \in C^\infty_0(D)$ with unit integral over $D$. The operator is then given for $\omega \in C^\infty(\wedge^k T^* D)$ by
    \[
    	T \omega = \int_D \varphi(y) K_y\omega \, \dd y,
    \]
    where for every $y \in D$, the operator $K_y \colon C^\infty(\wedge^k T^* D) \to C^{\infty}(\wedge^{k-1} T^* D)$ is given by
    \[
    	(K_y\omega)_x = \int_0^1 t^{k-1} (\omega_{tx + (1-t)y} \iprodl (x-y)) \, \dd t 
    \]
    for all $x \in D$.
    
    In \cite[(4.7)-(4.8)]{Iwaniec-Lutoborski}, it is shown that for every $\omega \in C^\infty(\wedge^k T^* D)$, $T \omega$ can be re-written in the form
    \begin{equation}\label{eq:Poincare_homotopy_structure}
    	(T\omega)_x = \int_D \omega_z \iprodl \zeta(z, x-z) \, \dd z
    \end{equation}
    for all $x \in D$, where the function $\zeta \colon \R^n \times (\R^n \setminus \{0\}) \to \R^n$ is given by
    \[
    	\zeta(z, v) = \left( \sum_{m=k}^n \binom{n-k}{m-k} \abs{v}^{n-m} \int_0^\infty s^{m-1} \varphi\left( z - s \frac{v}{\abs{v}} \right) \, \dd s \right) \frac{v}{\abs{v}^n}.
    \]
    In particular, the function $\zeta$ is of the form
    \begin{equation}\label{eq:zeta_decomp}
    	\zeta(z, v) = g_\zeta(z, v) \frac{v}{\abs{v}^n},
    \end{equation}
    where $g_\zeta \colon \R^n \times \R^n \to [0, \infty)$ is locally bounded in $\R^n \times \R^n$ and continuous in $\R^n \times (\R^n \setminus \{0\})$. 
    
    One may now use \eqref{eq:Poincare_homotopy_structure} as the basis for defining $T\omega$ for non-smooth $\omega$. In particular, if $\omega \in L^1(\wedge^k T^* D)$, then we have a point-wise estimate
    \begin{equation}\label{eq:Poincare_homotopy_pw_est}
    	\abs{\omega_z \iprodl \zeta(z, x-z)} \le C(n,k,D,\varphi) \frac{\abs{\omega_z}}{\abs{x-z}^{n-1}}
    \end{equation}
    for a.e.\ $x, z \in D$. A standard use of Young's inequality then yields that $T$ is a bounded operator from $L^q(\wedge^k T^* D)$ to $L^p(\wedge^{k-1} T^* D)$ when $p,q \in [1, \infty]$ with $q^{-1} < p^{-1} + n^{-1}$.
    
    \subsection{Proof of Proposition \ref{prop:compact_embedding}}
    
    Let $p,q \in [1, \infty]$ with $q^{-1} < p^{-1} + n^{-1}$. For now, we also suppose that $p \ge q$. Because the Kolmogorov-Riesz-Fr\'echet theorem is formulated for all of $\R^n$, we define an extension of $T$ to an operator $\tilde{T} \colon L^q(\wedge^k T^* D) \to L^{p}(\wedge^{k-1} T^* \R^n)$. For this, we fix a ball $B = \B^n(0, r)$ with $\overline{D} \subset B$, fix $\eta \in C_0(\R^n, [0, 1])$ with $\eta \equiv 1$ on $2B$, and $\spt \eta \subset 3B$, and define
    \begin{equation}\label{eq:xi_def}
    	\xi \colon \overline{D} \times (\R^n \setminus \{0\}) \to \R^n, \quad \xi(z, v) = \eta(v) \zeta(z, v).
    \end{equation}
    We then define
    \[
    (\tilde{T}\omega)_x = \int_D \omega_z \iprodl \xi(z, x-z) \, \dd z
    \]
    for all $x \in \R^n$. It follows that $(\tilde{T}\omega)_x = (T\omega)_x$ for all $x \in D$, since if $x, z \in D$, then $\abs{x - z} < 2r$, and consequently $\eta(x-z) = 1$. 
    
    Moreover, by \eqref{eq:zeta_decomp}, we now can express $\xi$ as
    \begin{equation}\label{eq:xi_decomp}
    	\xi(z, v) = g_\xi(z, v) \frac{v}{\abs{v}^n},
    \end{equation}
    where $g_\xi \colon \overline{D} \times \R^n \to [0,\infty)$ is continuous outside $\overline{D} \times \{0\}$, bounded, and supported in $\overline{D} \times 3B$. Notably, we conclude that
    \begin{equation}\label{eq:xi_pw_est}
    	\abs{\xi(z, x-z)} \le \norm{g_\xi}_{L^\infty(D\times\R^n)} \frac{\cX_{4B}(x)}{\abs{x-z}^{n-1}},
    \end{equation}
    for all $(z, x) \in D \times \R^n$, where $\cX_{A}$ denotes the characteristic function of a set $A$.
    
    We can in fact slightly adjust the decomposition \eqref{eq:xi_decomp} to make it better suited for our uses by moving a small power of $\abs{v}$ into $g_\xi$, which eliminates the discontinuity of $g_\xi$ at $\overline{D} \times \{0\}$ due to the boundedness of $g_\xi$. That is, for every $s > n-1$, we can write
    \begin{equation}\label{eq:xi_decomp_eps}
    	\xi(z, v) = g_{\xi, s}(z, v) \frac{v}{\abs{v}^{s+1}},
    \end{equation}
    where $g_{\xi, s} \colon \overline{D} \times \R^n \to [0,\infty)$ is continuous and supported in $\overline{D} \times 3B$. Consequently, $g_{\xi, s}$ is uniformly continuous.

    By \eqref{eq:xi_pw_est}, the assumption $p \ge q$, and Lemma \ref{lem:pseudo_young}, we get that $\tilde{T}$ is a bounded operator from $L^q(\wedge^k T^* D)$ into $L^p(\wedge^k T^* \R^n)$. However, the objective that requires more care is estimating $\smallnorm{\tilde{T}(\omega) - \tau_{h} \tilde{T}(\omega)}_{L^p(\R^n)}$, which will then yield compactness via Kolmogorov-Riesz-Fr\'echet for $1 \le p < \infty$ and Arzela-Ascoli for $p = \infty$. For this, we observe that
    \[
    \bigl(\tilde{T}\omega - \tau_h \tilde{T}\omega\bigr)_x = 
    \int_{D} \omega_z \iprodl (\xi(z, x-z) - \xi(z, x+h-z)) \, \dd z
    \]
    for all $x, h \in \R^n$. For convenience, we denote 
    \[
    	\Delta_h \xi_{z}(x) = \Delta_h \xi^x(z) = \xi(z, x-z) - \xi(z, x + h - z).
    \] 
    Thus, given the statement of Lemma \ref{lem:pseudo_young}, the missing piece we require is the following result on the limit behavior of $\Delta_h \xi_{z}$ and $\Delta_h \xi^x$.
    
    \begin{lemma}\label{lem:esssup_estimate_lemma}
    	Let $D \subset \R^n$ be a bounded, convex domain and let $\xi \colon \overline{D} \times \R^n$ be given by $\xi(\cdot,0)=0$ and \eqref{eq:xi_def}, with $k$, $\varphi$, $\eta$, $B$ fixed as described above. Then, for all $1 \le \gamma < n/(n-1)$, we have
    	\begin{align*}
    		&\lim_{h \to 0} \sup_{x \in \R^n} \norm{\Delta_h \xi^x}_{L^\gamma(D)} = 0 \qquad \text{and}\\
    		&\lim_{h \to 0} \sup_{z \in D} \norm{\Delta_h \xi_{z}}_{L^\gamma(\R^n)} = 0.
    	\end{align*}
    \end{lemma}
    \begin{proof}
    	Suppose that $h \in B$. Note that if $z \in D$, then $\Delta_h \xi_{z}(x) \ne 0$ only if $x \in 5B$. We fix $s \in (n-1, n/\gamma)$, noting that this interval is nonempty since $\gamma < n/(n-1)$, and apply decomposition \eqref{eq:xi_decomp_eps} with that $s$ to estimate that
    	\begin{multline*}
    		\abs{\xi(z, x-z) - \xi(z, x + h - z)}^\gamma \\
    		\le 2^\gamma \abs{g_{\xi, s}(z, x-z)}^\gamma  
    		\abs{\frac{x-z}{\abs{x-z}^{s+1}} - \frac{x+h-z}{\abs{x+h-z}^{s+1}}}^\gamma \\
    		+ \frac{2^\gamma}{\abs{x+h-z}^{s\gamma}} \abs{g_{\xi, s}(z, x-z) - g_{\xi, s}(z, x+h-z)}^\gamma
    	\end{multline*}
        for all $x\in 5B$ and $z\in D$. For convenience, we denote by $\Phi_s \colon \R^n \to \R^n$ the map given by $\Phi_s(0) = 0$ and $\Phi_s(v) = v/\abs{v}^{s+1}$ for $v \in \R^n \setminus \{0\}$. Then
    	\begin{multline}\label{eq:comp_1_est_dz}
    		\int_D \abs{g_{\xi, s}(z, x-z)}^\gamma 
    		\abs{\frac{x-z}{\abs{x-z}^{s+1}} - \frac{x+h-z}{\abs{x+h-z}^{s+1}}}^\gamma \, \dd z\\
    		\le \norm{g_{\xi, s}}_{L^\infty(\overline{D} \times \R^n)}^\gamma \norm{\Phi_s - \tau_h \Phi_s}^\gamma_{L^\gamma(6B)}
    	\end{multline}
    	for all $x \in 5B$, and similarly
    	\begin{multline}\label{eq:comp_1_est_dx}
    		\int_{5B} \abs{g_{\xi, s}(z, x-z)}^\gamma
    		\abs{\frac{x-z}{\abs{x-z}^{s+1}} - \frac{x+h-z}{\abs{x+h-z}^{s+1}}}^\gamma \, \dd x\\
    		\le \norm{g_{\xi, s}}^\gamma_{L^\infty(\overline{D} \times \R^n)} \norm{\Phi_s - \tau_h \Phi_s}^\gamma_{L^\gamma(6B)} 
    	\end{multline}
    	for all $z \in B$. Now, since $s\gamma < n$, there exists a $\delta > \gamma$ with $s\delta < n$. Hence, $\Phi_s \in L^\delta(7B)$, and consequently $\norm{\Phi_s - \tau_h \Phi_s}_{L^\delta(6B)} \le 2 \norm{\Phi_s}_{L^\delta(7B)}$ is uniformly bounded in $h$ since $\abs{h} < r$. Thus, we may use Lemma \ref{lem:pseudo_dominated_conv} on arbitrary sequences of $h_j \in B$ tending to 0 to conclude that
    	\begin{equation}\label{eq:comp_1_unif_est}
    		\norm{\Phi_s - \tau_h \Phi_s}_{L^\gamma(6B)} \xrightarrow[h \to 0]{} 0.
    	\end{equation}
    	
    	We then recall that $g_{\xi, s}$ is uniformly continuous. Thus, there exists a function $\Psi_s \colon \R^n \to [0, \infty]$ such that for all $x \in \R^n$ and $z \in D$, 
    	\begin{equation}\label{eq:comp_2_unif_est}
    		\abs{g_{\xi, s}(z, x-z) - g_{\xi, s}(z, x+h-z)} \le \Psi_s(h) \xrightarrow[h \to 0]{} 0.
    	\end{equation}
    	Thus, we may estimate
    	\begin{multline}\label{eq:comp_2_est_dz}
    		\int_{D} \frac{1}{\abs{x+h-z}^{s\gamma}} 
    		\abs{g_{\xi, s}(z, x-z) - g_{\xi, s}(z, x+h-z)}^\gamma \, \dd z\\
    		\le  \norm{\Phi_s}^\gamma_{L^\gamma(7B)} \Psi_s^\gamma (h)
    	\end{multline}
    	for all $x \in 5B$, and similarly
    	\begin{multline}\label{eq:comp_2_est_dx}
    		\int_{5B} \frac{1}{\abs{x+h-z}^{s\gamma}} 
    		\abs{g_{\xi, s}(z, x-z) - g_{\xi, s}(z, x+h-z)}^\gamma  \, \dd x\\
    		\le \norm{\Phi_s}^\gamma_{L^\gamma(7B)} \Psi_s^\gamma(h)
    	\end{multline}
    	for all $z \in D$.
    	Now, $\lim_{h \to 0} \sup_{x \in \R^n} \norm{\Delta_h \xi^x}_{L^\gamma(D)} = 0$ by \eqref{eq:comp_1_est_dz}, \eqref{eq:comp_1_unif_est}, \eqref{eq:comp_2_unif_est}, and \eqref{eq:comp_2_est_dz}. Similarly, $\lim_{h \to 0} \sup_{z \in D} \norm{\Delta_h \xi_{z}}_{L^\gamma(\R^n)} = 0$ by \eqref{eq:comp_1_est_dx}, \eqref{eq:comp_1_unif_est}, \eqref{eq:comp_2_unif_est}, and \eqref{eq:comp_2_est_dx}.
    \end{proof}
    
    We finally have the required ingredients which imply Proposition \ref{prop:compact_embedding}. 
    \begin{proof}[Proof of Proposition \ref{prop:compact_embedding}]
    	We note that for every $q \in [1, \infty]$, $q^{-1} < p^{-1} + n^{-1}$ is valid with $p = q$. Moreover, if $T \colon L^q(\wedge^k T^* D) \to L^p(\wedge^{k-1} T^* D)$ is compact, then $T \colon L^{q}(\wedge^k T^* D) \to L^{p'}(\wedge^{k-1} T^* D)$ is compact for all $p' \in [1, p]$ due to the boundedness of $D$ and H\"older's inequality. Thus, we may assume that $p \ge q$. We define $\gamma$ by $\gamma^{-1} = p^{-1} + 1 - q^{-1}$, noting that $p \ge q$ and $q^{-1} < p^{-1} + n^{-1}$ imply that $1 \le \gamma < n/(n-1)$.
    	
    	Let $\omega \in L^q(\wedge^{k} T^* D)$ with $\norm{\omega}_{L^q(D)} \le 1$. Suppose first that $p = \infty$. Then since the integrals of $1/\abs{x}^{\gamma(n-1)}$ over balls of a fixed size are uniformly bounded, \eqref{eq:xi_pw_est} and Lemma \ref{lem:pseudo_young} yield a uniform bound for $\smallabs{\tilde{T}\omega}$. Similarly, by Lemmas \ref{lem:pseudo_young} and \ref{lem:esssup_estimate_lemma}, we get a uniform equicontinuity estimate for $\tilde{T}\omega$. Moreover, all $\tilde{T} \omega$ are supported in $5B$. By Arzela-Ascoli, it follows that the image of the unit ball of $L^q(\wedge^{k} T^* D)$ under $\tilde{T}$ is precompact. Since restricting functions in $L^{p}(\wedge^{k-1} T^* \R^n)$ to $L^{p}(\wedge^{k-1} T^* D)$ preserves convergence of sequences, it also follows that the image of the unit ball of $L^q(\wedge^{k} T^* D)$ under $T$ is precompact, completing the proof of the case $p = \infty$.
    	
    	Suppose then that $1 \le p < \infty$. Similarly as above, by \eqref{eq:xi_pw_est}, Lemma \ref{lem:pseudo_young}, and the uniform bound on the integrals of $1/\abs{x}^{\gamma(n-1)}$ over balls of a fixed size, $\tilde{T}$ is a bounded operator from $L^q(\wedge^k T^* D)$ to $L^p(\wedge^{k-1} T^* \R^n)$. By Lemmas \ref{lem:pseudo_young} and \ref{lem:esssup_estimate_lemma}, we get an $L^p$-equicontinuity estimate for $\tilde{T}\omega$. Moreover, since all $\tilde{T} \omega$ are supported in $5B$, the family of all such $\tilde{T}\omega$ is trivially $L^p$-equitight. Theorem \ref{thm:Kolmogorov-Frechet-Riesz} then yields that the image of the unit ball of $L^q(\wedge^{k} T^* B)$ under $\tilde{T}$ is $L^p$-precompact. It follows similarly as before by restricting elements of $L^p(\wedge^{k-1} T^* \R^n)$ to $L^p(\wedge^{k-1} T^* D)$ that $T$ is a compact operator.
    \end{proof}
	
    
	\bibliographystyle{abbrv}
	\bibliography{sources}
	
\end{document}